\documentclass[oneside,9pt]{article} \topmargin -15mm \textwidth 160 true mm
\usepackage{t1enc}
\usepackage{graphicx}
\usepackage{float}
\usepackage{caption}
\usepackage{mdframed,lipsum}
\usepackage{mathtools} 
\usepackage{subcaption}
\usepackage{epstopdf}
\usepackage{mathrsfs}
\usepackage{multicol}
\usepackage{setspace}
\usepackage{amsmath, amssymb, amsfonts, amstext, amsthm,amscd}
\usepackage{relsize}
\usepackage{stmaryrd}
\usepackage{mathtools}
\DeclareSymbolFont{symbolsC}{U}{txsyc}{m}{n}
\DeclareMathSymbol{\Searrow}{\mathrel}{symbolsC}{117}
\usepackage[mathscr]{euscript}
\usepackage{enumerate, mathtools}
\usepackage{chngcntr}
\usepackage{tikz-cd}
\usepackage{tikz,color,soul, verbatim}
\usetikzlibrary{patterns}
\usepackage{wrapfig}
\usepackage{subcaption}
\usepackage{soul}
\usetikzlibrary{arrows.meta,arrows,shapes.geometric,mindmap,shapes}
\usepackage[english]{babel}
\usepackage[pdftex]{hyperref}
\newtheorem{theorem}{Theorem}[section]
\newtheorem{proposition}[theorem]{Proposition}
\newtheorem{lemma}[theorem]{Lemma}

\newtheorem{corollary}[theorem]{Corollary}
\theoremstyle{definition}
\newtheorem{definition}[theorem]{Definition}
\newtheorem{eg}[theorem]{Example}
\newtheorem{remark}[theorem]{Remark}

\newtagform{clai}{section}{Claim}

\numberwithin{equation}{section}

\numberwithin{mytheorem}{subsection}
\def \no {\noindent}
\def \R{\mathbb{R}}

\def \N{\mathbb{N}}
\def \C {\mathcal{C}}
\def \U {\mathcal{U}}
\def \rar{\rightarrow}

\def \D{\mathcal{D}}

\newcommand{\ring}[1]{\mathcal{R}_{#1}}

\newcommand{\sk}[2]{st_{#1}(#2)}
\newcommand{\Ld}[1]{Lk_{#1}(\Delta)}
\newcommand{\Mh}[1]{\mathcal{M}_H(\Delta(#1))}

\newcommand{\Cm}[1]{\C_{\operatorname{min}}(\Delta(#1))}
\newcommand{\Lk}[2]{Lk_{#1}(\Delta(#2))}

\newcommand{\cone}[2]{\operatorname{cone}_{#1}(#2)}
\newcounter{casenum}

\newcounter{subcasenum}

\title{Neural ring homomorphism preserves mandatory sets required for open convexity}
\author{Neha Gupta, Suhith K N\footnote{Suhith's research is partially supported by Inspire fellowship from DST grant IF190980.}}
\date{}
\begin{document}

	\maketitle
	\begin{abstract}
It has been studied in \cite{curto2017makes} that a neural code that has an open convex realization does not have any local obstruction relative to the neural code.	Further, a neural code $ \C $ has no local obstructions if and only if it contains the set of mandatory codewords, $ \C_{\min}(\Delta),$ which depends only on the simplicial complex $\Delta=\Delta(\C)$.  Thus if $\C \not \supseteq \C_{\min}(\Delta)$, then  $\C$ cannot be open convex.   However, the problem of constructing $ \C_{\min}(\Delta) $ for any given code $ \C $ is undecidable. There is yet another way to capture the local obstructions via  the homological mandatory set, $ \mathcal{M}_H(\Delta). $  The significance of $ \mathcal{M}_H(\Delta) $ for a given code $ \C $ is that $ \mathcal{M}_H(\Delta) \subseteq \C_{\min}(\Delta)$ and so $ \C $ will have local obstructions if $  \C\not\supseteq\mathcal{M}_H(\Delta). $  In this paper we study the affect on the sets $\C_{\min}(\Delta) $ and $\mathcal{M}_H(\Delta)$ under the action of various surjective elementary code maps. Further, we study the relationship between Stanley-Reisner rings of the simplicial complexes associated with neural codes of the elementary code maps.   Moreover, using this relationship, we give an alternative proof to show that $ \mathcal{M}_H(\Delta) $ is preserved under the elementary code maps.
	\end{abstract}
	\noindent{\bf Key Words}: {Neural code, Neural ring homomorphism, Stanley-Reisner ring, Simplicial complex, local obstruction}

\noindent{\bf 2010 Mathematics Subject Classification}:  Primary 52A37, 92B20. Secondary 54H99, 55N99.

\section{Introduction}

In 1971, O’Keefe and Dostrovsky \cite{burgess20142014,o1971hippocampus} discovered that a few cells in the rat's hippocampus fire only in specific locations of the rat's environment. These cells are named place cells.  Place cells are neurons that play a vital role in the rat's perception of space. Experimental results \cite{curto2017makes} show that the area where a specific place cell fired was approximately open convex sets in $\R^2$.  This gives the motivation to work with neural codes on $ n $ neurons and to figure out when they would correspond to open and convex regions of some $ \R^k $. A neural code on $n$ neurons, denoted by $ \C $, is a collection of subsets of the  set $ [n]=\{1,2,3,\dots,n\}. $ Given a collection of $ n $ sets $\U=\{U_1,\dots,U_n\}$ with $ U_i\subseteq\R^k, $  we correspond a neural code to $\U$, defined as  $ \C(\U)=\left\{\sigma\in [n] \ \Big\vert\ \displaystyle\bigcap_{j\in \sigma} U_j \backslash \displaystyle\bigcup_{i\not\in\sigma } U_i \not= \phi \right\} $ that represents these sets. Conversely, we say that a given neural code $ \C $ on $ n $ neurons is \textit{realizable} in $\R^k$, if there exists  $ \U=\{U_1,\dots,U_n\} $ with $ U_i\subseteq\R^k $ such that $ \C=\C(\U). $  In addition, if  $  U_i $'s are all convex sets then we call $ \C $ a convex code or convex realizable. Similarly, one can define an open convex and a closed convex code.  

The search for whether a given neural code has an open convex realization has led to a great development in this area. It has been established that not every neural code is open convex. However, it is also known that every neural code is at least convex realizable. This result was proved by Franke and Muthiah \cite{franke2018every} in 2019.  Cruz et al. \cite{cruz2019open} proved that max intersection-complete codes are open and closed convex. In \cite{suhith2021few}, we have shown the converse of Cruz's theorem also holds for a particular class of neural codes which we call \textit{doublet maximal codes}. 

Curto et al. \cite{curto2017makes} worked with local obstructions (a topological concept) relative to a  neural code.  They show that if a neural code is open convex, it has no local obstructions.  The question of whether a given neural code $\C$ has a local obstruction can be reduced to the question of whether it contains certain minimal code, $\Cm{\C}$, which is solely dependent on the simplicial complex, $\Delta=\Delta(\C),$ associated to the neural code $\C$. Definitions and details of these terms are discussed in the next section. 
Further in \cite{curto2013neural} the authors show that a neural code $ \C $ has no local obstructions if and only if $ \C\supseteq\C_{\min}(\Delta) . $ However, the problem of constructing $ \C_{\min}(\Delta) $ for any given code $ \C $ is undecidable. Another way to capture the local obstructions is via the homological mandatory set, $ \mathcal{M}_H(\Delta)$ (Refer Section  \ref{manset}). The significance of $ \mathcal{M}_H (\Delta) $ for a given neural code $ \C $ is that $ \mathcal{M}_H (\Delta) \subseteq \C_{\min}(\Delta)$ and $ \C $ will have local obstructions if $ \C\not\supseteq\mathcal{M}_H(\Delta). $ In \cite{curto2017makes}, the authors provide an algorithm to obtain $ \mathcal{M}_H (\Delta)$ for any given neural code using the Stanley-Reisner ideals \cite{miller2004combinatorial}.

Curto et al. \cite{curto2013neural} associated a ring $\ring{\C}$ (called neural ring) to a neural code $\C$. A special class of ring homomorphisms between two neural rings, called neural ring homomorphisms were introduced by Curto and Youngs \cite{curto2020neural}. They show a $1-1$ correspondence between a neural ring homomorphism and a code map between the corresponding  neural codes (these code maps will be a compositions of only five \textit{elementary code maps}). We have briefly explained these concepts in Section \ref{nrh}. 

The main objective of this paper is to see the action of the surjective elementary code maps on the set $ \Mh{\C}. $ We will also check whether $ \C_{\min}(\Delta(\C)) $ is preserved under these elementary code maps. Further, we study the relationship between Stanley-Reisner rings of the simplicial complexes associated with the domain and the codomain of the elementary code maps. Moreover, this relationship helps us give an alternative proof to show that $ \Mh{\C} $ is preserved under the elementary code maps. Before we explore further, let us introduce all the required terminologies.

\section{Preliminaries}
Let $\C$ be a neural code on $n$ neurons. From now onwards, we will write neural codes as just codes. Elements of a code are called \textit{codewords}. Given $\alpha\subseteq [n]$, then $\alpha$ is simply a subset of the set $[n]$. We now discuss three different forms of writing $\alpha$. First, the obvious one, is the \textit{set form}. For example, say, $\alpha$ is the subset $\{1,3\}$ of $[4].$ Second, we discuss the \textit{word form} for $\alpha\subseteq[n].$ In this form we ignore the parenthesis and commas, and write the elements of $\alpha$ as if they are letters of a word. For example $\alpha=\{1,3\}$ is $13$ in the word form.    Lastly, we define the \textit{binary form} for $ \alpha $ as the vector $ \alpha_1\alpha_2\cdots\alpha_n\in\{0,1\}^n $ such that $ \alpha_i=1 $ if and only if $ i\in\alpha. $ The same example $\alpha=\{1,3\} \subseteq [4]$ in binary form is $1010$.  Consequently, we can also represent a code in these three different forms.   Consider a code in set form $ \C=\{\emptyset,\{1\},\{1,2\},\{1,3\},\{2,4\}\} $ on $4$ neurons. In the word form  $ \C $ is represented as $\{\emptyset,1,12,13,24\} $. Finally, the binary form of this code will be $ \{0000,1000,1100,1010,0101\}. $ We may use both the word form and the binary form of the code interchangeably for the rest of the paper. 
\begin{remark}
	Let $ \omega\subseteq[n] $ and $ \omega=\omega_1\omega_2\cdots\omega_n $  be its binary form. As $ [n]\subset [n+1], $ we get $ \omega\subset[n+1]. $  Moreover, its binary form as a subset of $[n+1]$ will be $ \omega_1\omega_2\cdots\omega_n0. $ For sake of convenience, we write $ \omega0 $ when we want to consider  $ \omega $ on $ n+1 $ neurons. However, in the subset form, they are still the same. For example, the codeword $ 12 $ on 3 and 4 neurons has  $ 110 $ and $ 1100 $ as its binary form, respectfully. But as subsets, both represent the same subset $ \{1,2\} $. Furthermore, we say two sets $A\subseteq[n]$ and $B\subseteq[n+1]$ are equal in the subset form whenever for all $\omega\in A $ we have $\omega0\in B$ and for all $ \alpha'\in B$ we have $\alpha'=\alpha0$ with $\alpha\in A.$  
\end{remark}
\no We will now see some operations on the codewords of a code considering them in binary form. 
\begin{remark}[Operations on a codewords] \label{remoperations}
	Let $ \C $ be a code on $ n $ neurons and $ \alpha,\beta\in\C $ with $ \alpha=\alpha_1\alpha_2\cdots\alpha_n $ and $ \beta=\beta_1\beta_2\cdots\beta_n $ as their binary forms. We now discuss the following operations. 
	\begin{enumerate}
		\item  \textbf{Containment:} We say $ \alpha\subseteq \beta $ if and only if, for all $ i\in[n], \  \alpha_i\leq\beta_i. $
		\item \textbf{Intersection:} Let $ \sigma=\alpha\cap\beta $ with $ \sigma=\sigma_1\sigma_2\cdots\sigma_n, $ then for $ i\in[n] $ we have $\sigma_i=1$ if $ \alpha_i=\beta_i=1$, otherwise  $\sigma_i=0$. 		\item \textbf{Union:} Let $ \tau=\alpha\cup\beta $ with $ \tau=\tau_1\tau_2\cdots\tau_n $ then for $ i\in[n] $ we have  $\tau_i=0$ if $\alpha_i=\ \beta_i=0$ else $\tau_i=1$.
	\end{enumerate}
\end{remark}
\begin{definition}[Maximal codeword]
	Let $\C$ be a code. A codeword $ \sigma \in\C $ is said to be \textit{maximal} in $\C$ if it is not contained in any other codeword of $ \C. $ In other words, if there exists $ \tau \in\C $ such that $ \sigma\subseteq \tau, $ then $ \sigma=\tau. $ A maximal codeword is also called \textit{facet} of $\C$.
\end{definition}

\no Now we will introduce simplicial complexes and look into few of its properties.
\begin{definition}[Simplicial complex] \label{defsc} Let $ n\in \N. $
	A simplicial complex $ K $ is a nonempty collection of subsets of $ [n]=\{1,2,\dots,n\} $ which is closed under inclusion. In other words, if $ \sigma\in K $ and $ \tau\subseteq\sigma $ then $ \tau\in K. $ By definition, the empty set ($ \emptyset $) is always an element of any simplicial complex. An element of a simplicial complex $K$ is called simplex or a face.  A face $\sigma\in K$ is said to be maximal face or facet of $K$ if it is not contained in any other face of $ K. $ For any $ \sigma\in K $ we know that $ \sigma\subseteq [n] $ and we will represent the cardinality of $ \sigma $ by $ |\sigma| $. The dimension of a simplex $ \sigma\in K, $ denoted by $\dim \sigma$,  is given by $ \dim \sigma =  |\sigma|-1. $ Note that dimension of $ \emptyset  $ is $ -1 $.  The dimension of a simplicial complex $ K $ is given by  $ \max_{\sigma\in K}\ \{\dim \sigma\}. $ 
	
	\no A \textit{vertex} of a simplicial complex $K$ is a zero dimensional simplex. The set of all vertices of $ K $ is called the vertex set and is denoted by $ V_K. $ In set form, $ V_K=\{\sigma\in K\mid\dim \sigma =0\}$.  For clarity we will say $K$ is a simplicial complex on $n$ vertices. 
\end{definition}

\begin{definition}[Link] Let $ K $ be a simplicial complex and let $ \sigma\in K $. Then, the link of $ \sigma$ inside $ K $ is given by $ Lk_{\sigma}(K)=\{\omega\in K\ |\ \omega\cap\sigma=\emptyset\text{ and } \omega\cup\sigma\in K\}. $ Clearly,  $Lk_{\sigma}(K)$ is a sub-complex of $K$.
\end{definition}
\no Note that the difference between a code and a simplicial complex is that, a code may not preserve inclusions. Hence we can associate a simplicial complex to a code which shall include all the subsets of the codewords. 
\begin{definition}[Simplicial complex of a code $ \C $] Given a code $ \C $ on $ n $ neurons we denote the set $ \Delta(\C) $ as the simplicial complex of the code and it is given by $$\Delta(\C)=\{\alpha\subseteq[n]\mid\alpha\subseteq\beta, \text{ for some } \beta\in\C\}.$$ Note that by the above definition, we always have $ \C\subseteq \Delta(\C). $ \label{defscofc}
\end{definition} 
\begin{remark}
	We  note down two simple properties of the simplicial complex of a code. \label{lemimportant}
	\begin{enumerate}
		\item 	Let $ \C $ and $ \D $ be two codes on $ n $ neurons. If $ \C\subseteq\D \subseteq \Delta(\C) $ then $ \Delta(\C)=\Delta(\D). $ 
		\item 	Let $ \C $ be a code on $n$ neurons. Then $ \tau  $ is a facet of $ \C $ if and only if it is a facet of $ \Delta(\C). $ 
	\end{enumerate}
\end{remark}
\begin{definition}[Restriction of a simplicial complex $K$ to a set $\Gamma$] \label{defrestriction}
	Let $ K $ be a simplicial complex on $ n $ vertices and let $ \Gamma $ be a collection of some subsets of $ [n]. $  Then the restriction of $ K $ to $ \Gamma $ is defined as,
	\begin{align*}
		K|_\Gamma&=\{\alpha\in K\ \mid \alpha\subseteq\gamma \text{ for some } \gamma\in\Gamma\} \\
		&=\{\alpha\in K\ \mid\ \alpha\in\Delta(\Gamma)\} = K\cap \Delta(\Gamma).	
	\end{align*}
	Observe that if $ \Gamma=\{\gamma\}, $ a singleton set, then $ K|_\Gamma= \{\alpha\in K \mid \alpha\subseteq\gamma\} $ which we shall denote as $K|_\gamma.$  Note that $ K|_\Gamma $ is also a simplicial complex since it is an intersection of two simplicial complexes.
\end{definition}

\begin{eg} \label{exampledata}
	Consider the code $ \C=\{24,35,45,123\} \subseteq [6]$.  Then $ \Delta(\C)=\{1,2,3,4,5,12,13, \linebreak23,24,35,45,123\} $. For $ \sigma=2, \ \Lk{\sigma}{\C}=\{1,3,4,13\}. $  Consider $ \Gamma=\{35,45,123,6\}. $ Then $ \Delta(\C)|_\Gamma=\{1,2,3,4,5,12,13,23,35,45,123\}. $
\end{eg}

\no Now we define cone of a simplicial complex.
\begin{definition}[Cone of a simplicial complex]
	Let $ K $ be a simplicial complex and $ v $ be a new vertex not in  $ K. $ Then, the set given by $$\operatorname{cone}_v(K)=\{\sigma\cup v \mid\sigma\in K\}\cup K $$ is called the cone of $ K $ over the vertex $ v. $ 
\end{definition}

\no Next, we discuss when a link of a simplicial complex will become a cone over some vertex.

\begin{definition}[$f_\sigma$]\label{fsigma}
	Let $ K $ be a simplicial complex. For any $ \sigma \in K\ $ denote $ f_\sigma $ to be  the intersection of all facets of $ K $ containing $ \sigma. $
\end{definition} \no Curto et al. \cite{curto2017makes} gave the following lemma relating the link and cone depending upon $ f_\sigma. $
\begin{lemma}\cite[Lemma 4.5]{curto2017makes}
	Let $K$ be a simplicial complex and let $ \sigma\in K. $ Then $ \sigma=f_\sigma  $ if and only if $  Lk_{\sigma}(K) $ is not a cone of any sub-complex\footnote{$K'$ is called sub-complex of $K$, if $K'\subseteq K$ and $K'$ is also a simplicial complex.  }.  
\end{lemma}
\no We will explicitly use the following corollary of this above lemma.
\begin{corollary} \cite[Corollary 4.6]{curto2017makes}
	Let $K$ be a simplicial complex and let $ \sigma\in K $ be nonempty. If $ \sigma\not= f_\sigma, $ then $ Lk_{\sigma}(K)  $ is a cone of some sub-complex and hence contractible\footnote{A set is contractible if it is homotopy-equivalent to a point \cite{rotman2013introduction}.}. \label{corlemmacone}
\end{corollary}
\subsection{Mandatory Sets} \label{manset}

Cracking the information packed in a neural code is one of the central challenges of neuroscience. We focus our attention on open convex codes. Curto et al. \cite{curto2017makes} have studied local obstructions to open convexity, a notion first introduced in \cite{giusti2014no}. Curto et al. \cite{curto2017makes} show that a code with one or more local obstructions cannot be open convex. In one of their results they prove that
\begin{lemma}\cite[Lemma 1.2]{curto2017makes}
	If a code $\C$ is open convex, then $\C$ has no local obstructions.
\end{lemma}
This fact was first observed in \cite{giusti2014no}, using slightly different language. The converse, unfortunately, is not true. Perhaps, Curto et al. \cite{curto2017makes} gave a complete characterization of local obstructions in their main result, Theorem 1.3 in \cite{curto2017makes}.

By the contra positive of the above lemma, a code can not be open convex if it has a local obstruction. In fact, the question of whether a given code $\C$ has a local obstruction can be reduced to the question of whether it contains a certain minimal code, $\C_{\min}(\Delta)$, which depends only on the simplicial complex $\Delta=\Delta(\C)$. 
Let us give the definition of the set $\C_{\min}(\Delta)$, as given in \cite{curto2017makes}. They define it as follows:
$$\C_{\min}(\Delta) = \{ \sigma\in \Delta\mid\Ld{\sigma}  \text{ is non contractible } \}\hspace{1mm}\cup \hspace{1mm}\{\emptyset\}.$$

By Theorem 1.3 in \cite{curto2017makes}, any code $\C$ with simplicial complex $\Delta$ has no local obstructions precisely when $\C\supseteq \C_{\min}(\Delta)$.  Thus, the elements of $\C_{\min}(\Delta)$ are regarded as ``mandatory'' codewords with respect to open convexity, because they must all be included in the code $\C$ in order for it to be open convex. 

Note that by the definition of $\C_{\min}(\Delta)$, it depends only on the topology of the links of $\Delta$. The set $\C_{\min}(\Delta)$ contains all the facets of the simplicial complex $\Delta$. This is because for any facet $\sigma\in \Delta$, $Lk_\sigma(\Delta) =\emptyset$ which is non contractible, and thus $\C_{\min}(\Delta)$ contains all the facets of $\Delta$.

Computing $\C_{\min}(\Delta)$ is certainly simpler than finding all local obstructions. However, it is still difficult in general because determining whether or not a simplicial complex is contractible is undecidable \cite{tancer2016recognition}. For this reason,
we now consider a subset of $\C_{\min}(\Delta)$ corresponding to non contractible links that can be detected via homology.  The definition of this subset as given in \cite{curto2017makes} is
$$\mathcal{M}_H(\Delta) = \{\sigma\in\Delta \mid \text{dim } \widetilde{H}_i(Lk_\sigma(\Delta)), k) > 0 \text{ for some }i\},$$
where the $\widetilde{H}_i(.)$'s are reduced simplicial homology groups, and k is a field. Homology groups are topological invariants that can be easily computed for any simplicial complex, and reduced
homology groups simply add a shift in the dimension of $\widetilde{H}_0(.)$. This shift is designed so that for any contractible space $X$, dim $\widetilde{H}_i(X; k) = 0$ for all integers $i$. Observe that if $\widetilde{H}_i(Lk_\sigma(\Delta)) > 0$ for some $\sigma\in\Delta(\C)$ then $Lk_\sigma(\Delta)$ is not contractible. Hence we always have $\mathcal{M}_H(\Delta) \subseteq \C_{\min}(\Delta)$  for any code $\C$, and $\mathcal{M}_H(\Delta)$ is thus a subset of the mandatory codewords that must be included in any open convex code $\C$ with $\Delta(\C) = \Delta$. Note that while $\mathcal{M}_H(\Delta)$ depends on the choice of field $k$, $\mathcal{M}_H(\Delta)\subseteq \C_{\min}(\Delta)$ for any $k$. We can in fact algorithmically compute the subset, $\mathcal{M}_H(\Delta)$. We will later discuss how to compute $\mathcal{M}_H(\Delta)$ using machinery from combinatorial commutative algebra. 
\subsection{Stanley-Reisner (SR) ring}
In this section, we define some algebraic objects connected to combinatorics. We use these objects to compute  homologically detectable mandatory codewords, algebraically. Let $ k$ be a field and let us denote  $ \mathcal{P}(n) =k[x_1,\dots,x_n] $  the polynomial ring over $ k $ in $ n $ indeterminates. We fix this field for rest of the paper. So, we will not mention the field, and it will be implicitly considered to be $k$ only.  Denote $ x^\sigma =\prod_{i\in\sigma}{x_i} $ for any $ \sigma\subseteq [n]. $
\begin{definition}[SR ideal and SR ring] \cite[Definition 1.6]{miller2004combinatorial}
	The Stanley-Reisner ideal of a simplicial complex $ K $ on $ n $ vertices is the ideal $$I_K=\langle x^\tau \mid \tau \subseteq [n] \text{ and } \tau\notin K \rangle$$ generated by the monomials corresponding to the non-faces $ \tau $ of $ K. $ The Stanley-Reisner ring of $ K $ is the quotient ring $  \mathcal{P}(n)/I_K. $
\end{definition}
\begin{eg}
	Let $K=\{\emptyset,1,2\}$ be a simplicial complex on 2 vertices. As $12$ is the only face not in $K$ we have $I_K=\langle x_{1}x_{2}\rangle$. The Stanley-Reisner ring of $K$ is $ \mathcal{P}(2)/\langle x_{1}x_{2}\rangle=k[x_1,x_2]/\langle x_{1}x_{2}\rangle \cong k[x_1]\oplus k[x_2]. $
\end{eg}

\no We write the ideal generated by subsets of $\{x_1,\dots,x_n\}$ as $m^\tau= \langle x_i \mid i\in \tau\rangle$ corresponding to $\tau\subseteq [n]$. Next, we define the Alexander dual of a monomial ideal. 
\begin{definition} [Alexander dual of a monomial ideal]
	Let $I=\langle x^{\sigma_{1}},\dots,x^{\sigma_{r}}\rangle $, for some $ \sigma_i \subseteq [n] $	and $r\in \N$. Then the Alexander dual of $I$  is $$I^*=\mathfrak{m}^{\sigma_1}\cap \cdots\cap \mathfrak{m}^{\sigma_r}.$$
\end{definition} 
\begin{definition}[Alexander dual of a simplicial complex]\cite[Definition 1.35]{miller2004combinatorial}
	If $ K$ is a simplicial complex and $ I = I_K$  is its Stanley–Reisner ideal, then the simplicial complex $ K^*, $ Alexander dual to $ K, $ is deﬁned by $ I_{K^*} = (I_{K})^*$. 
\end{definition}
Let $\C$ be a code with $\Delta=\Delta(\C)$. We now discuss how to compute the entire set $\mathcal{M}_H(\Delta)$ algebraically, via a minimal free resolution of an ideal built from $\Delta$. More specifically, $\mathcal{M}_H(\Delta)$ can also be computed with the following formulation:  
\begin{equation}\label{eqmh}
	\mathcal{M}_H(\Delta)	=\{\sigma\in\Delta\mid \beta_{i,\bar{\sigma}}( \mathcal{P}(n)/I_{\Delta^*})>0 \text{ for some } i >0\},
\end{equation} where $ \beta_{i,\bar{\sigma}}( \mathcal{P}(n)/I_{\Delta^*}) $ are  the Betti numbers of a minimal free resolution of the ring $  \mathcal{P}(n)/I_{\Delta^*}. $ This definition of $\mathcal{M}_H(\Delta)$ comes as a direct consequence of Hochster’s formula \cite[Corollary 1.40]{miller2004combinatorial}, which is
$$ \dim \widetilde{H}_i(\Ld{\sigma},k)= \beta_{i+2,\bar{\sigma}}( \mathcal{P}(n)/I_{K^*}).$$ See \cite{miller2004combinatorial} for more details on Alexander duality, Hochster's formula and SR ideal.

Moreover, the subset of mandatory codewords $\mathcal{M}_H(\Delta)$ can be easily computed using
some computational algebra software, such as Macaulay2 \cite{M2}. In fact, we have used Macaulay2 in an example in Section \ref{projection}, to compute the elements of $\mathcal{M}_H(\Delta).$   

\subsection{Simple and Strong Homotopy type}
\label{simpleandstrongtype}
One essential question in algebraic topology has always been when are two simplicial complexes homotopic?  We discuss a few homotopy preserving moves mentioned in R Ehrenborg's paper   \cite{ehrenborg2006topology}  and J Barmak's book on ``Algebraic topology of finite topological spaces and applications'' \cite{barmak2011algebraic}.  We first discuss simple homotopy.

Let $ K $ be a simplicial complex. Then star of $ \sigma \in K$  is given by $ \sk{\sigma}{K}=\{\tau\in K\mid \sigma\subseteq \tau\}. $ A pair of simplices $ \sigma\subsetneq\tau $  in $ K$ is called a \textit{free face pair} if $ \sk{\sigma}{K}=\{\sigma,\tau\}.$  
\begin{definition}[Elementary collapse and simple homotopy type] Let $ \sigma\subsetneq \tau $ be a free face pair in the simplicial complex $ K.  $  Removal of $ \sigma $ and $ \tau  $ from $ K$ is called an \textit{elementary collapse} of $ K. $ Elementary collapse from $ K $ to $ K\backslash \{\sigma,\tau\} $ is denoted by $ K \searrow^e K\backslash\{\sigma,\tau\}.$ Two simplicial complexes  $ K$ and $ K'$ are said to be of same \textit{simple homotopy type} if there is a finite sequence $ K=K_1,K_2,\dots,K_n=K'$ of simplicial complexes such that  for each index  $ i $ at least one of $ K_i $ and $ K_{i+1} $ may be obtained from the other by an elementary collapse. Simple homotopy between $ K$ and $ K' $ is denoted by $ K\searrow K'. $
\end{definition}

J Barmak \cite{barmak2011algebraic} shows that if there is a simple homotopy between $ K $ and $ K' $ then $ K$ is homotopic to $ K'. $ Next we discuss strong homotopy. 
Let $ v $ be a vertex of a simplicial complex $K $ then $ v $ is said to be \textit{dominated} if there exists a vertex   $ v'$  of $ K $ such that $ \Ld{v}=\cone{v'}{K'}, $ for some sub-complex $ K' $ of $ K. $ Denote $ K \smallsetminus v $ to be the \textit{induced simplicial complex} induced by the vertices of $ K $ except $ v. $

\begin{remark}\cite[Remark 2.2]{barmak2012strong}
	A vertex v is dominated by a vertex $ v'\not=v $ if and only if every facet that contains $ v $ also contains $ v' $. \label{remarkstrongcollapse}
\end{remark} 

\begin{definition}[Strong elementary collapse and strong homotopy type] \cite[Definition 5.1.1]{barmak2011algebraic}
	We say there is a strong elementary collapse from $ K $ to $ K\smallsetminus v $ if $ v $ is a dominated vertex of $ K. $ It is denoted by $ K \Searrow^{e} K\smallsetminus v.$ There is a strong collapse from $ K $ to a sub-complex $ K' $ if there exist a sequence of elementary strong collapses that starts in  $ K $ and ends in $ K'.$  In this case we write $ K\Searrow K' $.  The inverse of strong collapse is strong expansion. Two finite complexes $ K $ and $ K' $ have the same strong homotopy type if there is a sequence of strong collapses and/or strong expansions that starts in $ K $ and ends in $ K' $. 
\end{definition}
J Barmak \cite{barmak2011algebraic} proves that strong homotopy type is a particular case of simple homotopy and that if there exists strong homotopy between $ K  $ and $ K' $ then there exists a simple homotopy too. That is $ K\Searrow K' \implies K \searrow K'.$ Moreover, as simple homotopy implies homotopy,  we get $ K $ is homotopic to $ K' $ whenever $ K\Searrow K'. $  We will use all these concepts later to prove our main results.

\subsection{Neural ring homomorphisms} \label{nrh}
Curto et al. \cite{curto2013neural} associated a ring  $ \ring{\C} $ for a given code $ \C $, and  called it the neural ring associated to the code $\C$. For a given code $ \C, $ the neural ring is obtained as $ \ring{\C}=\mathbb{F}_2[y_1,\dots,y_n]/I_{\C}, $ where $ I_\C=\{f\in\mathbb{F}_2[x_1,x_2,\dots,x_n] | f(c)=0 \text{ for all } c\in \C\}. $  Further, Curto and Youngs \cite{curto2020neural}, studied the ring homomorphisms between two neural rings. They showed that there is a 1-1 correspondence between code maps $ q:\C\rar \D $ and the ring homomorphisms $ \phi:\ring{\D}\rar \ring{\C}. $ The associated code map of the ring homomorphism $ \phi $ is denoted as $ q_\phi $.  They also demonstrated that $ |\C|=|\D| $ is the only condition under which $ \ring{\C} \cong \ring{\D} $. This means that the ring homomorphisms ignore the nature of codewords that are present in the code and only take its cardinality into account. So, to capture the essence of the codewords of a code and not just its cardinality, \cite{curto2020neural} added some additional pertinent conditions to the ring homomorphisms and named these maps neural ring homomorphisms. We will now give detailed definition of a neural ring homomorphisms.
\begin{definition}[Neural ring homomorphism]\cite[Definition 3.1]{curto2020neural}
	Let $ \C  $ and $ \D $ be codes on $ n $ and $ m $ neurons respectively, and let $ \ring{\C}= \mathbb{F}_2[y_1,\dots,y_n]/I_{\C} $ and $ \ring{\D}= \mathbb{F}_2[x_1,\dots,x_m]/I_{\D}  $ be the corresponding neural rings.  A ring homomorphism $ \phi:\ring{\D}\rar \ring{\C} $ is a neural ring homomorphism if $ \phi(x_j)\in\{y_i\mid i \in [n]\} \cup \{0,1\} $ for all $ j\in [m].$  We say that a neural ring homomorphism $ \phi $ is a neural ring isomorphism if it is a ring isomorphism and its inverse is also a neural ring homomorphism. 
\end{definition}    
\no Further, Curto and Youngs showed that given any neural ring homomorphism $ \phi $, the associated code map $ q_\phi $ can only be a composition of five elementary code maps. We state this theorem below.
\begin{theorem}\cite[Theorem 3.4]{curto2020neural} \label{thmnycc1}
	A map $ \phi:\ring{\D}\rar \ring{\C} $ is a neural ring homomorphism if and only if $ q_\phi $ is a composition of the following elementary code maps: \\ 1. Permutation, 2. Adding a trivial neuron (or deleting a trivial neuron), 3. Duplication of a neuron (or deleting a neuron that is a duplicate of another), 4. Neuron projection (or deleting a not necessarily trivial neuron), 5. Inclusion (of one code into another).
	\\ \no 	Moreover, $ \phi $ is a neural ring isomorphism if and only if $ q_\phi $ is a composition of maps $ (1)-(3). $   
\end{theorem}
Now that we have discussed all the prelims required for the paper, we quickly go through the overall structure of the paper. In section 3 we show that if a neural ring homomorphism, $ \phi:\ring{\D}\to\ring{\C} $ is a isomorphism then $ q_\phi(\Mh{\C}) =\Mh{\D}.$ We will prove these results using the fact that if the map $ \phi $ is an isomorphism, then $ q_\phi $ is a composition of elementary code maps (1)-(3) of Theorem \ref{thmnycc1}. Further, we also know if the map $ \phi $ is just a surjection, then $ q_\phi $ is a composition of $ (1)-(4) $ of Theorem \ref{thmnycc1}. In section 4 we show that if $ \phi $ is  surjective then $ \Mh{\D}\subseteq q_\phi(\Mh{\C}). $  In both these sections we will also check when is $ \C_{\min}(\Delta(\C)) $ preserved.  So, now we consider all these elementary code maps individually and show how for every code $\C,\ $ $\mathcal{M}_H(\Delta(\C)) $ and $\Cm{\C}$ behave under the action of elementary code maps (1)-(4) mentioned in Theorem \ref{thmnycc1}. 
\section{Neural ring isomorphisms }
Let $ \phi:\ring{\D}\to\ring{\C} $ be a neural ring isomorphism. Then from Theorem \ref{thmnycc1}, we know that the associated code map is a composition of permutation, adding/deleting a trivial neuron or adding/deleting a duplicate neuron. We will examine all these elementary code maps individually to to see their action on sets, $\Mh{\C}$ and $\Cm{\C}$.

\subsection{Permutation}
We will denote the symmetric group on $n$ elements as $S_n$. Let $ \C $ be a code on $ n $ neurons and obtain $ \C' $ by permuting the neurons of each element of $ \C $ by a permutation   $ \gamma\in S_n.$ In other words, an element  $ \sigma=\sigma_{1}^{}\sigma_{2}^{}\cdots\sigma_{n}\in\C $ will be mapped to $ \sigma'=\sigma_{\gamma(1)}^{}\sigma_{\gamma(2)}^{}\cdots\sigma_{\gamma(n)}^{} $ in $ \C'$. Define $ q:\C\to\C' $  as $ \sigma\mapsto \sigma'=\sigma_{\gamma(1)}^{}\sigma_{\gamma(2)}^{}\cdots\sigma_{\gamma(n)}^{}.$ We further extend this map to $ \Delta(\C)$ and abusing the notation we call the extended map also as $ q. $  Let $ \alpha,\ \beta\in\Delta(\C) $ such that  $\alpha\subseteq\beta,  $ and $ \alpha',\beta' $ be the respective images of $\alpha,\ \beta $ under the map $ q $. Observe that  $ \alpha'\subseteq\beta'. $  In other words, if $ \alpha\subseteq\beta\in\Delta(\C) $ then $ q(\alpha)\subseteq q(\beta). $ We know that for any given $ \alpha\in\Delta(\C), $ there always exists $ \beta\in\C $ such that $ \alpha\subseteq\beta. $ Then $ q(\alpha)\subseteq q(\beta)\in\C'\subseteq\Delta(\C'), $ and as $ \Delta(\C') $ is a simplicial complex, $ q(\alpha)\in \Delta(\C'). $ Thus the extended map $ q:\Delta(\C)\to\Delta(\C') $ is well defined. We define a map $ p:\Delta(\C')\to\Delta(\C) $ using $ \gamma^{-1}, $ which sends  $ \alpha'=\alpha'_1\alpha'_2\cdots\alpha'_n $ to $ \alpha =
\alpha'_{\gamma^{-1}(1)}
\alpha'_{\gamma^{-1}(2)}
\cdots\alpha'_{\gamma^{-1}(n)}.$  Just like the map $q,$ the map $p$ is also well defined. 
It is clear that  $ q\circ p=1_{\Delta(\C')} $ and similarly $ p\circ q=1_{\Delta(\C)}. $ Therefore $ q $ is bijective and hence $ q(\Delta(\C))=\Delta(\C'). $ Henceforth we denote the map $ p $ as $ q^{-1}. $ One can consider $ q^{-1} $ to be the permutation map corresponding to $ \gamma^{-1}\in S_n. $ Also, as $ \emptyset=\underbrace{00\cdots0}_{n\text{ times}} $ it is clear that $ q(\emptyset)=\emptyset. $

\no By the nature of the permutation map $ q:\C\to\C' $ we observe that the structure of any link inside $ \Delta(\C) $ is  preserved under the permutation map. In other words, for $\sigma\in\Delta(\C)  $, $ \Lk{\sigma}{\C} \cong \Lk{q(\sigma)}{\C'}. $ This leads us to the following results. 

\begin{theorem}
	Let $ \C $ be a code on $n$ neurons and $ q: \C \to\C'$ be the permutation map corresponding to  $ \gamma\in S_n $ described above. Then \begin{enumerate}
		\item $ q(\Mh{\C})=\Mh{\C'}$
		\item  $ q(\Cm{\C})=\Cm{\C'}.$
	\end{enumerate} \label{thprem}
\end{theorem}

\begin{proof}
	Proof of first part: Let $ q(\sigma)\in q(\Mh{\C}), $ where $ \sigma\in\Mh{\C} $. By definition, we get $ \dim \widetilde{H}_i( \Lk{\sigma}{(\C)})>0 $ for some $ i. $ Also, by previous observation  $ \widetilde{H}_i(\Lk{\sigma}{\C})\cong \widetilde{H}_i(\Lk{q(\sigma)}{\C'}) $ so  $  \dim \widetilde{H}_i(\Lk{q(\sigma)}{(\C')})>0 $ and therefore $ q(\sigma)\in\Mh{\C'}. $ So,  $ q(\Mh{\C}) \subseteq\Mh{\C'}.$ Conversely, since $ q $ is surjective, we have for any $ \sigma'\in\Mh{\C'} \subseteq \Delta(\C')$ there exists $ \sigma\in \Delta(\C) $ such that $ q(\sigma) =\sigma'$. As $ \sigma'\in\Mh{\C'} $ we have $  \dim\widetilde{H}_i(\Lk{q(\sigma)}{\C'}) >0$.  Also, as   $ \widetilde{H}_i(\Lk{\sigma}{\C})\cong \widetilde{H}_i(\Lk{q(\sigma)}{\C'}) $  we have $ \dim \widetilde{H}_i(\Lk{\sigma}{\C}) >0.$ This implies $ \sigma\in\Mh{\C} $ which gives $ q(\sigma)\in q(\Mh{\C}) $. Therefore $  \Mh{\C'}\subseteq q(\Mh{\C}). $  Hence the result.   
	
	\no Proof of Second part: For any $\sigma\in\Delta(\C)  $ as $ \Lk{\sigma}{\C} \cong \Lk{q(\sigma)}{\C'}$ we have $ \Lk{\sigma}{\C} $ is not contractible if and only if $ \Lk{q(\sigma)}{\C'} $ is not contractible. Using this fact we can construct the proof of this part similar to first part.
\end{proof}

\no \textbf{Relation between SR rings of $ \Delta(\C) $ and $ \Delta(\C'). $} \\
\no  In the case of permutation, it is clear that $\C $ and $\C’$ are bijective sets. Even $\Delta(\C)$ and $\Delta(\C’)$
are isomorphic as simplicial complexes. We further show that even the SR rings of $\Delta(\C)$
and $\Delta(\C’)$ are isomorphic.
\begin{theorem}
	Let $ \C $ be a code on $n$ neurons and $ q:\C\to\C' $ be the permutation map corresponding to $ \gamma\in S_n $ described above. Then  SR rings of $ \Delta(\C) $ and $ \Delta(\C') $ are isomorphic. \label{SRpermutation}
\end{theorem}
\begin{proof}
	Given $\gamma\in S_n$, define a map $ \Gamma: \mathcal{P}(n)\to  \mathcal{P}(n)$ on the generators $ \{x_i\}_{i\in[n]} $ by $ x_i\mapsto x_{\gamma(i)}. $ Then define $ \Phi_\Gamma: \mathcal{P}(n)/ I_{\Delta(\C)^*} \to  \mathcal{P}(n)/ I_{\Delta(\C')^*}, $ by $ f+I_{\Delta(\C)^*}\to \Gamma(f) + I_{\Delta(\C')^*} $ for any $ f\in  \mathcal{P}(n). $ Observe that the bijection between $ \Delta(\C) $ and $ \Delta(\C') $ gives $ I_{\Delta(\C')^*}=\Gamma(I_{\Delta(\C)}^*) $ making the map $ \Phi_\Gamma $ well defined. Moreover, this map $ \Phi_\Gamma  $ will be a ring isomorphism. Hence the result.
\end{proof}

\begin{remark}
	Now the proof of Theorem \ref{thprem} can also be obtained using Theorem \ref{SRpermutation} and Equation \ref{eqmh} which is the algebraic form of $ \mathcal{M}_H(\Delta(\C)). $
\end{remark}
\subsection{Adding a trivial neuron}
Let $ \C $ be a code on $ n $ neurons and obtain $\C'  $ by adding a trivial `on' neuron, i.e., for any $ \sigma=\sigma_1^{} \sigma_2^{}\cdots\sigma_{n}^{}\in\C $, we will form $ \sigma'=\sigma_1^{} \sigma_2^{}\cdots\sigma_{n}^{}1 $ in $ \C'. $  We denote $ \sigma'=\sigma1. $ Let $ q:\C\to\C' $ be a code map given by $ \sigma \mapsto \sigma1. $  Further, extend the map $ q $ to $ \Delta(\C) $ and abusing the notation call the extended map also as $ q. $ 
Given any $ \alpha\in\Delta(\C) $ there exists $\beta\in\C $ such that $ \alpha\subseteq\beta. $ So, $ q(\beta)=\beta1\in\C'\subseteq\Delta(\C'). $ Also, as $ \alpha1\subseteq\beta1 $ we get $ q(\alpha)\subseteq q(\beta)\in \Delta(\C'). $ Since, $ \Delta(\C') $ is a simplicial complex, $ q(\alpha)\in\Delta(\C'). $ Therefore $ q(\Delta(\C))\subseteq\Delta(\C').$ So, the extension map $ q:\Delta(\C) \rar\Delta(\C')$ is well defined.  Moreover, as $ \C\subseteq \Delta(\C) $ and $ q(\C)=\C' $ we get $ \C'=q(\C)\subseteq q(\Delta(\C)). $ Thus $  \C'\subseteq q(\Delta(\C))\subseteq\Delta(\C') $ and by Remark \ref{lemimportant} we have \begin{equation}\label{eqaddqdelta}
	\Delta(q(\Delta(\C)))=\Delta(\C'). 
\end{equation}
Now, we establish series of lemmas before we prove our main theorem.
\begin{lemma}
	Let $ \C $ be a code on $n$ neurons and $ q:\C\to \C' $ be the 	 map described above. If $ \alpha'\in\Delta(\C')$ such that $ \alpha'=\alpha1 $ for some $\alpha\subseteq[n]$ then $ \alpha'\in q(\Delta(\C)). $ \label{lemreq}
\end{lemma}
\begin{proof}
	Let $ \alpha'\in\Delta(\C'), $ from Equation \ref{eqaddqdelta}, $\alpha'\in\Delta(q(\Delta(\C))). $ Then there exists $ \beta'\in q(\Delta(\C)) $ such that $ \alpha'\subseteq\beta'. $ This implies, there exists $ \beta\in\Delta(\C) $ such that $ q(\beta) =\beta'$. Since $ q(\beta)=\beta1, $ we have $ \alpha1=\alpha'\subseteq\beta'=\beta1.$ So, $ \alpha\subseteq\beta\in\Delta(\C)$ and as $ \Delta(\C) $ is a simplicial complex, $ \alpha\in\Delta(\C). $ Also, $ q(\alpha)=\alpha1=\alpha' $ we get $ \alpha'\in q(\Delta(\C)). $  Hence the result.  
\end{proof}
\begin{lemma}
	Let $ \C $ be a code on $n$ neurons and $ q:\C\to \C' $ be the map described above. Then for any $ \sigma\in \Delta(\C) $ we get $ \Lk{\sigma}{\C}=\Lk{q(\sigma)}{\C'}. $ \label{lemlinktrv}
\end{lemma}
\begin{proof}
	Let $ \omega\in \Lk{\sigma}{\C}$ then we show that $ \omega=\omega0\in\Lk{q(\sigma)}{\C'}. $ As  $ \omega\cap\sigma=\emptyset $ we get $ \omega0\cap\sigma1=\emptyset. $ This implies  $ \omega0\cap q(\sigma)=\emptyset. $ Now, we are left with showing $ \omega0\cup q(\sigma)\in \Delta(\C'). $ Let $\tau= \omega\cup\sigma. $ Then  $ \tau\in\Delta(\C)$ as $ \omega\in\Lk{\sigma}{\C}. $ Consider $\omega0\cup q(\sigma) =\omega0\cup\sigma1=\tau1= q(\tau) \in q(\Delta(\C))\subset\Delta(\C').$ Therefore $ \omega\in\Lk{q(\sigma)}{\C'}. $ So, $ \Lk{\sigma}{\C}\subseteq\Lk{q(\sigma)}{\C'}. $
	
	Let $ \lambda=\lambda_1\lambda_2\cdots\lambda_{n+1}\in\Lk{q(\sigma)}{\C'} $ then $\emptyset= \lambda\cap q(\sigma)=\lambda\cap\sigma1=\lambda_1^{}\lambda_2^{}\cdots\lambda_n^{}\lambda_{n+1}^{}\cap\sigma_{1}^{}\dots\sigma_{n}1. $ So, $\lambda_{n+1}=0.  $ We rewrite $ \lambda=\lambda0 $ and we show that $ \lambda\in\Lk{\sigma}{\C}. $ It is clear that $ \lambda\cap\sigma=\emptyset. $ Let $ \tau=\lambda\cup\sigma$ then we need to show that $ \tau\in\Delta(\C). $ As $ \lambda0\in\Lk{q(\sigma)}{\C'} $ we get $ \lambda0\cup q(\sigma)=\lambda0\cup\sigma1=\tau1\in\Delta(\C') $. By Lemma \ref{lemreq} we get that $ \tau1\in q(\Delta(\C))$ and we end up with $ \tau\in\Delta(\C). $  Therefore $ \lambda\in\Lk{\sigma}{\C} $ and thus $ \Lk{q(\sigma)}{\C'}\subseteq\Lk{\sigma}{\C}. $ Hence the proof.
\end{proof}
\begin{lemma}
	Let $ \C $ be a code on $ n $ neurons and $ q:\C\rar\C' $ be the map described above. Then $ \Mh{\C'}\subseteq  \Cm{\C'}\backslash\{\emptyset\}\subseteq q(\Delta(\C)).$ \label{lemaddimagelink}
\end{lemma}
\begin{proof}
	
	We will first show that $ \Cm{\C'}\backslash\{\emptyset\}\subseteq q(\Delta(\C)) $. Let  $ \alpha'\in\Cm{\C'}\backslash\{\emptyset\}. $ To the contrary, let us suppose  $\alpha'\not\in q(\Delta(\C)).  $  Since $ \alpha'\in\Cm{\C'} $ we have $ \alpha'\in \Delta(\C') $. Then by Lemma \ref{lemreq}, $ \alpha'=\alpha0   $ where $ \alpha=\alpha_1^{}\alpha_2^{}\cdots\alpha_n^{} $. Now we show that $ \Lk{\alpha'}{\C'} $ is contractible.
	
	We will establish $ \alpha
	'\not= f_{\alpha'} $(Refer Definition \ref{fsigma}). Note that if $ \tau'=\tau_1\dots\tau_n\tau_{n+1} $ is a facet of $ \Delta(\C') $ then by equation \ref{eqaddqdelta},  $ \tau'\in q(\Delta
	(\C))$. Therefore $\tau_{n+1}=1. $ Hence the $ n+1^{\text{th}} $ neuron in $ f_\alpha' =1$. But as $ \alpha'=\alpha0 $ we get that $ \alpha'\not= f_{\alpha'}. $ Now using Corollary \ref{corlemmacone}, $ \Lk{\alpha'}{\C'} $ is contractible.  Thus by definition of $ \Cm{\C'}\backslash\{\emptyset\} $,  we have $ \alpha' \notin \Cm{\C'}\backslash\{\emptyset\} $ which is a contradiction. Hence we have our claim that $ \alpha'\in q(\Delta(\C)) $ and $ q(\alpha)=\alpha1=\alpha'. $ Therefore $ \Cm{\C'}\backslash \{\emptyset\}\subset q(\Delta(\C')) $
	
	We already know that $ \Mh{\C'}\subseteq \Cm{\C'} $. Moreover, $ \emptyset\notin\Mh{\C'} $ as $ \C' $ is a cone of $ \C $ over $v_{n+1}=\underbrace{00\cdots0}_{n\text{ times }}1$. Therefore $ \Mh{\C'} \subseteq \Cm{\C'}\backslash \{\emptyset\}. $ Hence the result. 
\end{proof}

\no The next result shows that $ \Mh{\C} $ is preserved under the adding trivial `on' neuron map. Moreover, we show that non-empty elements of $ \Cm{\C} $ are preserved. 
\begin{theorem}
	Let $ \C $ be a code on $n$ neurons and $ q:\C\to \C' $ be the map described above. Then \begin{enumerate}
		\item $ q(\Mh{\C})=\Mh{\C'}.$
		\item \begin{enumerate}
			\item If $ \Delta(\C) $ is contractible then $ q(\Cm{\C}\backslash\{\emptyset\})=\Cm{\C'}\backslash\{\emptyset\}.$
			\item If $ \Delta(\C) $ is not contractible then $ q(\Cm{\C})\subsetneq\Cm{\C'}$. \\Moreover, $ q(\Cm{\C})=\Cm{\C'}\backslash\{\emptyset\}. $
		\end{enumerate}
	\end{enumerate} \label{thadd}
\end{theorem}

\begin{proof}
	Proof of part 1:
	Let $ q(\sigma)\in q(\Mh{\C}), $ then $ \sigma\in\Mh{\C} $ and by definition we get $ \dim \tilde{H}_i (\Lk{\sigma}{(\C)})>0 $ for some $ i. $ By Lemma \ref{lemlinktrv}, $  \dim \tilde{H}_i (\Lk{q(\sigma)}{(\C')})>0 $ and therefore $ q(\sigma)\in\Mh{\C'}. $ So, $ q(\Mh{\C}) \subseteq\Mh{\C'}.$
	
	Conversely, let $ \alpha'\in \Mh{\C'}, $ then by Lemma \ref{lemaddimagelink} there exists an $ \alpha $ such that $ q(\alpha)=\alpha'. $ Now, we will show that $\alpha'=q(\alpha)\in q(\Mh{\C}). $ As $ \alpha'\in\Mh{\C'} $ we get $ \dim \tilde{H}_i (\Lk{q(\alpha)}{\C'}) >0$ for some $ i $ and by Lemma \ref{lemlinktrv} we get that $ \dim \tilde{H}_i (\Lk{\alpha}{\C})>0 $ for the same $ i. $ Therefore $ \alpha\in\Mh{\C}  $ and  $ q(\alpha)\in q(\Mh{\C}). $ Hence the result.
	
	\no	Proof of part 2: The proof part (a) is similar to the proof of part 1 and follows from Lemma  \ref{lemlinktrv} and Lemma \ref{lemaddimagelink}. The proof part (b) follows from the fact that $ \Lk{q(\emptyset)}{\C'} =\Lk{\emptyset}{\C}=\Delta(\C)$ (Refer Lemma \ref{lemlinktrv}) and $ \Delta(\C) $ being not contractible. So, $ q(\emptyset)\in \Cm{\C'}. $ Hence the proof. \end{proof}	
\no \textbf{Relationship between SR rings of $ \Delta(\C) $ and $ \Delta(\C') $}.
\begin{lemma}
	Let $ \C $ be a code on $n$ neurons and $ q:\C\to \C' $ be the map described above. Then the  SR ideal of $ \Delta(\C) $ is subset to the SR ideal of $ \Delta(\C') $. \label{lemsradd}
\end{lemma}
\begin{proof}
	Let $ x^\tau\in I_{\Delta(\C)} $ be a generator of $ I_{\Delta(\C)} $ then $ \tau\notin\Delta(\C) $. Suppose if $ \tau\tau_{n+1}\in\Delta(\C') $ then  $ \tau0\in\Delta(\C') $. By Equation \ref{eqaddqdelta}, there exists $\beta'\in q(\Delta(\C))$ such that $ \tau0\subseteq \beta '$.  Further, there exists $ \beta\in\Delta(\C) $ such that $ q(\beta)=\beta1=\beta' $. Moreover, $ \tau\subseteq \beta. $ This gives a contradiction that $ \tau\in\Delta(\C) $. Therefore, $ \tau0\notin\Delta(\C'). $ Hence $ x^{\tau0}=x^\tau $ is a generator of $ I_{\Delta(\C')} $. Therefore, $I_{\Delta(\C)}\subseteq I_{\Delta(\C')}\subseteq  \mathcal{P}(n+1)$.
\end{proof}

\begin{theorem}
	Let $ \C $ be a code on $n$ neurons and $ q:\C\to\C' $ be the map described above. Then  SR ring of $ \Delta(\C) $ is embedded in SR ring of $ \Delta(\C') $. \label{SRaddtrv}
\end{theorem}
\begin{proof}
	Define $ \phi:  \mathcal{P}(n)/I_{\Delta(\C)}\to  \mathcal{P}(n+1)/I_{\Delta(\C')} $ by mapping $ f+I_{\Delta(\C)} \mapsto f+I_{\Delta(\C')}. $ Observe that due to Lemma \ref{lemsradd} this map is a well-defined injective ring homomorphism. Hence we get the result.
\end{proof}
\no We were able to draw some conclusion about $ \Mh{\C} $ using its SR ring in  previous case as we did obtain some concrete relationship between SR rings. However, in this case we only have one sided containment. So we do not have enough information to draw any conclusion about $ \Mh{\C} $ using its SR ring.  

\begin{remark}\textbf{Adding a trivial off neuron:}
	%
	\label{remaddtrvoff}
	Let $ \C $ be a code on $ n $ neurons and obtain $\C'  $ by adding a trivial `off' neuron to each element of $ \C, $ i.e., for any $ \sigma=\sigma_1^{} \sigma_2^{}\cdots\sigma_{n}^{} \in\C $ form $ \sigma'=\sigma_1^{} \sigma_2^{}\cdots\sigma_{n}^{}0 $ in $ \C' $. We denote $ \sigma'=\sigma0. $ Let $ q:\C\to\C' $ be a code map given by $ \sigma \mapsto \sigma0. $  Further extend the map $ q $ to $ \Delta(\C) $ and abusing the notation call the extended map also as $ q. $ 
	Similar discussion as before leads us to the fact that the extension map $ q:\Delta(\C) \rar\Delta(\C')$ is well defined. 
	%
	
	Note that if $ q(\sigma)=\sigma'=\sigma0, $ then as subsets of $ [n+1] $, both $\sigma$ and $\sigma'$ represent the same subset, i.e., the elements $\sigma$ of $\Delta(\C)$ and $\sigma'$ of $\Delta(\C')$ have the same representation when seen as subsets of $[n+1].$ Thus $\Delta(\C) $ and $\Delta(\C') $ represent the same collection of subsets of $[n+1]$. Accordingly, even $\mathcal{M}_H$ sets are in $1-1$ correspondence, i.e.,  $\mathcal{M}_H(\Delta(\C))\cong\Mh{\C'}$ via the map $q$.  Thus $ q(\Mh{\C})=\Mh{\C'}$, whenever we are adding a trivial off neuron.  Similarly, one can observe that $ \Cm{\C} $ will also be preserved. 
	
	\no \textbf{Relationship between SR rings of $\Delta(\C)$ and $ \Delta(\C'). $}
	We observe that $ I_{\Delta(\C')^*} $ is generated by generators of $ I_{\Delta(\C)^*}$ and $x_{n+1} $. Therefore the SR ring of $ \Delta(\C') $ is $  \mathcal{P}(n+1)/ I_{\Delta(\C')^*}= k[x_1,\dots,x_{n+1}]/I_{\Delta(\C)^*}\cup\langle x_{n+1}\rangle\cong k[x_1,\dots,x_n]/I_{\Delta(\C)^*}= \mathcal{P}(n)/I_{\Delta(\C)^*},$ which is the SR ring of $ \Delta(\C). $ Hence in this case of adding a trivial off neuron, the SR rings are isomorphic. 
	Therefore we directly get that $ \Mh{\C} $ is preserved. 
	
\end{remark}

\subsection{Duplicating a neuron}
Let $ \C $ be a code on $ n $ neurons and obtain $\C'  $ by duplicating the first neuron and adding it to the last position. Thus for any  $ \sigma= \sigma_1^{} \sigma_2^{}\dots\sigma_{n}^{}\in\C$ we obtain $ \sigma' $ by duplicating $ \sigma_1 $, i.e., $ \sigma'=\sigma_1^{} \sigma_2^{}\cdots\sigma_{n}^{}\sigma_{1}^{}. $  We denote $ \sigma'=\sigma\sigma_{1}. $ Let $ q:\C\to\C' $ be a code map given by $ \sigma \mapsto \sigma\sigma_{1}. $ Extend the map $ q $ to $ \Delta(\C) $, abuse the notation and call the extended map as $ q. $
Given any $ \alpha\in\Delta(\C) $ there exists $\beta\in\C $ such that $ \alpha\subseteq\beta$. So, $ q(\beta)=\beta\beta_1\in\C'\subseteq\Delta(\C'). $ Also, as $ \alpha\alpha_1\subseteq\beta\beta_1 $ we get $ q(\alpha)\subseteq q(\beta)\in \Delta(\C'). $ Since, $ \Delta(\C') $ is a simplicial complex, $ q(\alpha)\in\Delta(\C'). $ Therefore  $ q(\Delta(\C))\subseteq\Delta(\C').$ So, the extended map $ q:\Delta(\C) \rar\Delta(\C')$ is well defined.  Moreover, as $ \C\subseteq \Delta(\C) $ and $ q(\C)=\C' $ we get $ \C'=q(\C)\subseteq q(\Delta(\C)). $ Thus $  \C'\subseteq q(\Delta(\C))\subseteq\Delta(\C') $ and by Remark \ref{lemimportant}, \begin{equation}\label{eqdupqdelta}
	\Delta(q(\Delta(\C)))=\Delta(\C'). 
\end{equation}
\begin{lemma} \label{lemlinkduplicate} 
	Let $ \C $ be a code on $n$ neurons and $ q:\C\to\C' $ be the duplicating code map described above. If $ \alpha'\in\Delta(\C') $ such that $ \alpha'=\alpha\alpha_1 $ for some $ \alpha=\alpha_1\alpha_2\cdots\alpha_n$ then $ \alpha'\in q(\Delta(\C)). $ 
	\begin{proof}
		Let $ \alpha'=\alpha\alpha_1\in\Delta(\C'), $ from Equation \ref{eqdupqdelta} we have  $\alpha'\in\Delta(q(\Delta(\C))). $ Then there exists some  $ \beta'\in q(\Delta(\C)) $ such that $ \alpha'\subseteq\beta'. $ Let $ \beta'=\beta\beta_1 $ for some $ \beta\in\Delta(\C) $. Then  $ \alpha\alpha_1=\alpha'\subseteq\beta'=\beta\beta_1 $, implies $ \alpha\subseteq \beta\in\Delta(\C). $ But as $ \Delta(\C) $ is a simplicial complex, $ \alpha\in\Delta(\C). $ Also, $ q(\alpha)=\alpha\alpha_1=\alpha'. $ Therefore $ \alpha'\in q(\Delta(\C)). $
	\end{proof}
\end{lemma}

\begin{lemma}\label{remlinkdup}
	Let $q:\C\to\C'$ be the duplicating code map described above. Let $\alpha=\alpha_1\dots\alpha_n$  be some codeword of $\C$.  If $\alpha0\in\Delta(\C')$ with $\alpha_1=1$, or if $\alpha1\in\Delta(\C')$ then the element $1\alpha_2\dots\alpha_n1\in\Delta(\C').$ 
\end{lemma}

\begin{proof}
	Let $ \alpha'=\alpha1\in\Delta(\C') $. Then by definition there exists  $ \beta'=\beta1\in\C'$ where  $\beta=\beta_1\cdots\beta_n\in\C $ such that $ \alpha1=\alpha'\subseteq\beta'=\beta1. $ Since $ q:\C\rar\C' $ is a bijection $ \beta'= q(\beta)$.  Further $\beta1=q(\beta)=\beta\beta_1$. This implies $ \beta_1=1. $ So, $ \alpha\alpha_1\subseteq1\alpha_2\cdots\alpha_n1\subseteq\beta1\in\Delta(\C') $. Hence $ 1\alpha_2\cdots\alpha_n1\in\Delta(\C'), $ whenever $ \alpha1\in\Delta(\C'). $
	
	Let $ \alpha'=\alpha0\in\Delta(\C'), $with $ \alpha_1=1. $ Then by definition there exists $ \beta'=\beta'_1\cdots\beta'_n\beta_{n+1}\in\C' $ such that $ \alpha'\subseteq \beta'. $ As $ \alpha_1=1 $ we must have $ \beta'_1=1. $ Moreover, as $ q:\C\to\C' $ is a bijection, there exists $ \beta\in\C $ such that $ q(\beta)=\beta\beta_1=\beta'. $ So, $ \beta_i=\beta'_i $ for all $ i\in[n]. $ Therefore  $ \beta'=\beta1\in\C'\subseteq\Delta(\C'). $ In addition as $ \alpha1\subseteq\beta1\in\Delta(\C'), $ and as $ \Delta(\C') $ is a simplicial complex we have $ \alpha1\in\Delta(\C'), $ whenever $ \alpha0\in\Delta(\C') $ with $ \alpha_1=1. $ 
\end{proof}

\begin{lemma} \label{lemlinkdup}
	Let $ \C $ be a code on $n$ neurons and $ q:\C\to\C' $ be the duplicating map described above. Then for any $ \sigma=\sigma_1\cdots\sigma
	_n\in \Delta(\C) $ we have $$ \Lk{q(\sigma)}{\C'}=\begin{cases}
		\Lk{\sigma}{\C} \quad & \text{if } \sigma_{1}=1\\
		\Delta(q(\Lk{\sigma}{\C})) \quad & \text{if } \sigma_{1}=0.
	\end{cases} $$
\end{lemma}
\begin{proof}
	Let $ \sigma=\sigma_1\cdots\sigma_n\in\Delta(\C).  $ We will discuss 2 cases depending on the values that $\sigma_1$ can take.\\
	\textbf{Case 1: }{$ \sigma_{1}=1 $.}\\ In this case $ q(\sigma)=\sigma1. $  This case is similar to Lemma  \ref{lemlinktrv} and we get the result.\\ 
	\textbf{Case 2: }{$ \sigma_{1}=0 $.} \\In this case $ q(\sigma)=\sigma0. $ Let $ \alpha'\in\Delta(q(\Lk{\sigma}{\C})), $ then by definition there exists $ \omega'\in q(\Lk{\sigma}{\C}) $ such that $ \alpha'\subseteq\omega'. $  We will show that $ \omega'\in\Lk{q(\sigma)}{\C'} $. As $ \alpha'\subseteq\omega' $ and $ \Lk{q(\sigma)}{\C'} $ being a simplicial complex, $ \alpha'\in\Lk{q(\sigma)}{\C'} $  . Since $ \omega'\in q(\Lk{\sigma}{\C})$,   $ \omega\in\Lk{\sigma}{\C} $ such that $ \omega'= q(\omega)=\omega\omega_1. $ 
	As $ \omega\in\Lk{\sigma}{\C} $ we have $ \omega\cap\sigma=\emptyset $. So, $ \omega'\cap q(\sigma)= \omega\omega_1\cap\sigma0=\emptyset. $ Next, let $\tau= \omega\cup\sigma.$ Then  $\tau\in\Delta(\C) $ as $ \omega\in\Lk{\sigma}{\C} $. Note that,  $ \tau_1=\omega_1, $ as $ \sigma_1=0. $ We know that $ q(\tau)=\tau\tau_1\in q(\Delta(\C))\subseteq \Delta(\C'). $ Also, $ \omega'\cup q(\sigma)= \omega\omega_1\cup\sigma0=\tau\omega_1=\tau\tau_1\in\Delta(\C'). $  Hence $ \omega'\in\Lk{q(\sigma)}{\C'}. $  Therefore $ \alpha'\in\Lk{q(\sigma)}{\C'}. $ Thus  $ \Delta(q(\Lk{\sigma}{\C}))\subseteq\Lk{q(\sigma)}{\C'}. $ 
	
	\no Conversely, consider $ \alpha'=\alpha\alpha_{n+1}\in\Lk{q(\sigma)}{\C'}, $ where $ \alpha=\alpha_1\cdots\alpha_n. $ To show that $ \alpha'\in\Delta(q(\Lk{\sigma}{\C})) $ We will work with two sub-cases:
	
	\no\textbf{Case 2a:} $ \alpha_1=\alpha_{n+1}. $\\
	In this case we have $ \alpha'=\alpha\alpha_1 $. We first show that $ \alpha\in\Lk{\sigma}{\C}. $ We have $ \alpha\cap\sigma=\emptyset $ as $ \alpha'\cap q(\sigma)=\alpha\alpha_1\cap\sigma0=\emptyset. $ Let $ \alpha\cup\sigma=\tau $  with $ \tau=\tau_1^{}\cdots\tau_n^{}. $ Note that $ \tau_1=\alpha_1 $ as $ \sigma_1=0. $ Consider $ \alpha'\cup q(\sigma)=\alpha\alpha_1\cup\sigma0=\tau\alpha_1=\tau\tau_1 $. Since $ \alpha'\in\Lk{q(\sigma)}{\C'} $ we have $ \alpha'\cup q(\sigma) \in\Delta(\C').$ So, $ \tau\tau_1\in\Delta(\C'). $ By Lemma $ \ref{lemlinkduplicate}, $  $ \tau\tau_1\in q(\Delta(\C)) $ and $ q(\tau)=\tau\tau_1. $ This implies $ \tau\in\Delta(\C). $ Therefore $ \alpha\in\Lk{\sigma}{\C}. $ Furthermore, $ \alpha\alpha_1=q(\alpha)\in q(\Lk{\sigma}{\C})\subseteq\Delta(q(\Lk{\sigma}{\C})). $ So, $ \alpha'\in\Delta(q(\Lk{\sigma}{\C})). $ 
	
	\no\textbf{Case 2b:} $\alpha_1\not= \alpha_{n+1}.$\\ 
	In this case we will show that $ 1\alpha_2\cdots\alpha_n1\in q(\Lk{\sigma}{\C}) $. Also as $ \alpha'=\alpha\alpha_{n+1}\subseteq 1\alpha_2\cdots\alpha_n1. $ Then by definition of simplicial complex, $ \alpha'\in\Delta(q(\Lk{\sigma}{\C})) $. Let us denote   $ \omega=1\alpha_2\alpha_{3}\cdots\alpha_{n}$. We first show that $ \omega\in\Lk{\sigma}{\C}. $ We know that $ \emptyset=\alpha'\cap q(\sigma)= \alpha\alpha_{n+1}\cap\sigma0. $ However, as $ \sigma_1=0 $ we get $ \omega0\cap\sigma0= 1\alpha_2\cdots\alpha_n0\cap0\sigma_2\cdots\sigma_n0=\emptyset. $ Thus 
	$ \omega\cap\sigma=\emptyset $. Let $ \alpha\cup\sigma=\alpha_1\cdots\alpha_n\cup0\sigma_2\cdots\sigma_n:=\beta $. Note that $ \beta_1=\alpha_1 $ as $ \sigma_1=0. $ Next, $ \omega\cup\sigma=1\alpha_2\cdots\alpha_n\cup0\sigma_2\cdots\sigma_n=1\beta_2\beta_3\cdots\beta_n. $ Let $ \tau=1\beta_2\beta_3\cdots\beta_n. $ To show that $ \tau\in\Delta(\C). $ Consider $  \alpha'\cup q(\sigma)=\alpha\alpha_{n+1}\cup\sigma0=\beta\alpha_{n+1}.$  
	Since $ \alpha'\in\Lk{q(\sigma)}{\C'} $ we have $ \alpha'\cup q(\sigma)\in\Delta(\C'). $ So, $ \beta\alpha_{n+1}\in\Delta(\C'). $  If $ \alpha_{n+1}=1 $ then we have $ \alpha_1=\beta_1=0 $ and $ \beta1\in\Delta(\C'). $ Moreover, by Lemma \ref{remlinkdup}, $ \tau\tau_1=1\beta_2\cdots\beta_n1\in\Delta(\C'). $ Else if $ \alpha_{n+1}=0 $ then $ \beta0\in\Delta(\C') $ with $ \beta_1=\alpha_{1}=1. $ As  $ \beta_1=1 $ we have $ \beta=\tau$ and $ \tau0\in\Delta(\C'). $ So, by Lemma \ref{remlinkdup}, $ \tau\tau_1=\tau1\in\Delta(\C'). $ Finally in either case  $ \tau\tau_1\in\Delta(\C'). $ By Lemma
	\ref{lemlinkduplicate},  $ \tau\tau_1\in q(\Delta(\C)) $ with $ \tau\tau_1=q(\tau). $ This implies that $ \tau\in\Delta(\C). $ Thus  $ \omega\in\Lk{\sigma}{\C}. $ Further $ q(\omega)=\omega\omega_1=\omega1\in q(\Lk{\sigma}{\C}). $ Note that $ \alpha'=\alpha\alpha_{n+1}=\alpha_1\alpha_2\cdots\alpha_n\alpha_{n+1}\subseteq1\alpha_2\cdots\alpha_n1=\omega1\in q(\Lk{\sigma}{\C}).  $ By definition of the simplicial complex we have $ \alpha'\in\Delta(q(\Lk{\sigma}{\C})). $
	
	\no  Therefore in either case $ \alpha'\in \Delta(q(\Lk{\sigma}{\C})). $ Thus $ \Lk{\sigma}{\C}\subseteq \Delta(q(\Lk{\sigma}{\C})).$ Hence the proof.
\end{proof}
\begin{remark} \label{remarkrelink}
	
	Let $ \sigma=\sigma_1\cdots\sigma_n $ with $ \sigma_1=0 $ and  $ \alpha'=\alpha\alpha_{n+1}\in\Lk{q(\sigma)}{\C'}, $ where $ \alpha=\alpha_1\cdots\alpha_n. $ In this remark we will show $\alpha\in\Lk{\sigma}{\C}$.
	By Lemma \ref{lemlinkdup}, $ \alpha'\in\Delta(q(\Lk{\sigma}{\C})) $. This implies that there exists $ \beta'\in q(\Lk{\sigma}{\C}) $ such that $ \alpha'\subseteq\beta'. $ In addition, $ \beta'=\beta\beta_1 $ with $ \beta\in\Lk{\sigma}{\C}. $ Also, as $ \alpha\alpha_{n+1}=\alpha'\subseteq\beta'=\beta\beta_1 $ we get $ \alpha\subseteq\beta\in\Lk{\sigma}{\C}. $ As $ \Lk{\sigma}{\C} $ is a simplicial complex, $ \alpha\in\Lk{\sigma}{\C}$ whenever $ \sigma_1=0 $ and $ \alpha'=\alpha\alpha_{n+1}\in \Lk{q(\sigma)}{\C} $. 
\end{remark}

\no We establish the following proposition using strong homotopy type (Refer to Section \ref{simpleandstrongtype}). 

\begin{proposition}\label{remlinkduplicate}
	Let $ \C $ be a code on $ n $ neurons and $ q:\C\to\C' $ be a duplicating map described above. If $ \sigma\in\Delta(\C) $ then  $ \Lk{q(\sigma)}{\C'} $ is homotopic to $ \Lk{\sigma}{\C} $. 
\end{proposition}

\begin{proof}
	Let $ \sigma=\sigma_{1}\sigma_{2}\cdots\sigma_{n} \in\Delta(\C). $ We have two cases:\\
	\textbf{Case 1:}  $ \sigma_1=1 $. \\
	In this case  by Lemma \ref{lemlinkdup}, $ \Lk{\sigma}{\C}=\Lk{q(\sigma)}{\C'} $. Hence the result follows trivially.\\
	\textbf{Case 2:} $ \sigma_1=0. $\\
	\textbf{Case 2a:} If for all $ \omega\in\Lk{\sigma}{\C} $ we have $ \omega_1=0$, then in that case  $ q(\Lk{\sigma}{\C})=\{q(\omega)\mid\omega\in\Lk{\sigma}{\C}\}=\{\omega0\mid\omega\in\Lk{\sigma}{\C}\}. $ Also, as subsets of $ [n+1],\ \omega=\omega0. $ So, we get $ \{\omega0\mid\omega\in\Lk{\sigma}{\C}\}=\Lk{\sigma}{\C}. $ Therefore $ q(\Lk{\sigma}{\C})=\Lk{\sigma}{\C}. $ Further $ \Delta(q(\Lk{\sigma}{\C})) =\Delta(\Lk{\sigma}{\C})=\Lk{\sigma}{\C}$ as $ \Lk{\sigma}{\C} $ is itself a  simplicial complex. By Lemma \ref{lemlinkdup},  $ \Lk{q(\sigma)}{\C'}=\Delta(q(\Lk{\sigma}{\C})) $. Therefore $\Lk{q(\sigma)}{\C'}=\Lk{\sigma}{\C}.  $ Hence the result. \\
	\textbf{Case 2b:} If there exists $ \omega\in\Lk{\sigma}{\C} $ such that $ \omega_1 =1 $, then  $ \omega1=q(\omega)\in q(\Lk{\sigma}{\C}).  $ Let $ v_1=1\underbrace{00\cdots0}_{n\text{ times }} $ and $v_{n+1}=\underbrace{00\cdots0}_{n\text{ times }}1$.  Then   $v_1,v_{n+1}\subseteq \omega1\in q(\Lk{\sigma}{\C}). $ Therefore by definition of simplicial complex we get $ v_1,v_{n+1}\in\Delta(q(\Lk{\sigma}{\C})). $ Also, by Lemma \ref{lemlinkdup}, $v_1, v_{n+1}\in\Lk{q(\sigma)}{\C'}. $  We will now show $ v_{n+1} $ is dominated by $ v_1 $ in $ \Lk{q(\sigma)}{\C'}. $ 
	
	Let $ \tau'\in\Lk{q(\sigma)}{\C'} $ be a facet containing $ v_{n+1}. $  Then by Lemma \ref{lemlinkdup} we have $ \tau'\in\Delta(q(\Lk{(\sigma)}{\C})) $ and this implies that there exists $ \beta'\in q(\Lk{\sigma}{\C}) $ such that $ \tau'\subseteq\beta'. $ Now $ \beta'\in q(\Lk{\sigma}{\C})\subseteq \Delta(q(\Lk{\sigma}{\C}))=\Lk{q(\sigma)}{\C'} $ so $ \beta'\in\Lk{q(\sigma)}{\C'}. $ Thus maximality of $ \tau' $ implies $ \tau'=\beta'\in q(\Lk{\sigma}{\C}). $
	This implies there exists $ \tau\in\Lk{\sigma}{\C} $ such that $ \tau'=q(\tau)=\tau\tau_1. $ 
	Moreover, as $ v_{n+1}\subseteq\tau' =\tau\tau_1 $ we must have $ \tau_1=1. $ This implies $ v_1\subseteq\tau'. $ Therefore every facet containing $ v_{n+1} $ contains $ v_1. $ By Remark  \ref{remarkstrongcollapse}, $ v_{n+1} $ is dominated by $ v_1. $ Hence $ \Lk{q(\sigma)}{\C'}\Searrow\Lk{q(\sigma)}{\C'}\smallsetminus v_{n+1}. $ 
	
	Next we will show that $ \Lk{q(\sigma)}{\C'}\smallsetminus v_{n+1} =\Lk{\sigma}{\C}. $ By Remark \ref{remarkrelink} if  $ \alpha'=\alpha\alpha_{n+1}\in\Lk{q(\sigma)}{\C'} $ then $ \alpha\in\Lk{\sigma}{\C}. $ Consider any $ \alpha'\in\Lk{q(\sigma)}{\C'} \smallsetminus v_{n+1} $ then $ \alpha'=\alpha0 $ and we know  $ \alpha'=\alpha0=\alpha\in\Lk{\sigma}{\C}. $ Thus  $ \Lk{q(\sigma)}{\C'} \smallsetminus v_{n+1} \subseteq\Lk{\sigma}{\C}.$ Next, given $ \alpha\in\Lk{\sigma}{\C} $ we get $ \alpha\alpha_1=q(\alpha)\in q(\Lk{\sigma}{\C}). $ We know that $ \alpha0\subseteq\alpha\alpha_1, $ so  $ \alpha0\in\Delta(q(\Lk{\sigma}{\C}))=\Lk{q(\sigma)}{\C'}. $  Thus $ \Lk{\sigma}{\C}\subseteq\Lk{q(\sigma)}{\C'}. $
	Moreover, as $ v_{n+1}\notin   \Lk{\sigma}{\C} $ we get $ \Lk{\sigma}{\C}\subseteq \Lk{q(\sigma)}{\C'}\smallsetminus v_{n+1}. $ Thus $ \Lk{q(\sigma)}{\C'}\smallsetminus v_{n+1} =\Lk{\sigma}{\C}. $ Therefore $ \Lk{q(\sigma)}{\C'}\Searrow\Lk{\sigma}{\C}. $ Hence we have $ \Lk{q(\sigma)}{\C'} $ is homotopic to $ \Lk{\sigma}{\C} $ and we get the result. 
\end{proof}

\begin{lemma}
	Let $ \C $ be a code on $ n $ neurons and $ q:\C\rar\C' $ be the duplicating map described above. Then $ \Mh{\C'}\subseteq \Cm{\C'}\subseteq q(\Delta(\C)). $  \label{lemlinkneeded}
\end{lemma}
\begin{proof}
	We will first show $ \Cm{\C'}\subseteq q(\Delta(\C))$.
	Note that $ q(\emptyset)=\emptyset                                       $. So $ \emptyset\in\Cm{\C'} $ has a pre-image. 
	Therefore $\emptyset\in q(\Delta(\C))$. Let us choose some non-empty $ \sigma'=\sigma_1\cdots\sigma_{n+1}\in\Cm{\C'} $. Denote $ \sigma_1\dots\sigma_n $ to be $ \sigma $. Assume that $ \sigma'=\sigma\overline{\sigma_1} $, where $ \overline{\sigma_1} $ is a conjugate \footnote{Conjugate of $\sigma_1$: $\overline{\sigma_1}=0$ or $1$ whenever $\sigma_1=1$ or $ 0$, respectively. } of $ \sigma_1. $ Then we claim that $ \sigma'\not= f_{\sigma'} $ (Refer Definition \ref{fsigma}). 
	The proof of claim is similar to the one in Theorem \ref{lemaddimagelink}. By equation \ref{eqdupqdelta} we observe that any facet of $ \Delta(\C') $ is an element of $ q(\Delta(\C)) $. Therefore the $ n+1^\text{th} $ neuron of the facet (and in turn for $ f_\sigma' $) is $ \sigma_1. $ This gives us that $ \sigma'\not = f_{\sigma'}.  $ So, by Corollary \ref{corlemmacone} $ \Lk{\sigma'}{\C'} $ is contractible. Therefore by definition we get that $ \sigma'\notin \Cm{\C'}. $ This is a contradiction to our first assumption. So the assumption that $ \sigma'=\sigma\overline{\sigma_1} $ is wrong. Hence $ \sigma'=\sigma\sigma_1.  $ Thus  by Lemma \ref{lemlinkduplicate}, $ \sigma'\in q(\Delta(\C)). $ So $ \Cm{\C'}\subseteq q(\Delta(\C)) $, and $ \Mh{\C'}\subseteq q(\Delta(\C)). $ Hence the proof.
\end{proof}

\begin{theorem}
	Let $ \C $ be a code on $ n  $ neurons and $ q:\C\to\C' $ be the code map described above. Then\begin{multicols}{2}
		\begin{enumerate}
			\item $ q(\Mh{\C})= \Mh{\C'} $
			\item $ q(\Cm{\C})= \Cm{\C'}. $
		\end{enumerate}.
	\end{multicols}  \label{thdup}
\end{theorem}

\begin{proof}
	For the proof of first part, let us consider $ q(\sigma)\in q(\Mh{\C}), $ for some $ \sigma\in\Mh{\C} $. By definition, $ \dim \widetilde{H}_i (\Lk{\sigma}{\C})>0 $ for some $ i. $ Since homotopy preserves homological dimensions, we have from  Proposition  \ref{remlinkduplicate}, $ \dim\widetilde{H}_i(\Lk{\sigma}{\C})= \dim\widetilde{H}_i(\Lk{q(\sigma)}{\C'}). $ So,  $  \dim \widetilde{H}_i (\Lk{q(\sigma)}{\C'})>0 $ and therefore $ q(\sigma)\in\Mh{\C'}. $ Hence $ q(\Mh{\C}) \subseteq\Mh{\C'}.$ Next, by Lemma \ref{lemlinkneeded}  $ \Mh{\C'} \subseteq q(\Delta(\C))$. Therefore for every $ \sigma'\in\Mh{\C'} $ there exists $ \sigma\in \Delta(\C) $ such that $ q(\sigma) =\sigma'.$ Now as $ \sigma'\in\Mh{\C'} $,  $ \dim\widetilde{H}_i(\Lk{\sigma'}{\C'})=\dim\widetilde{H}_i(\Lk{q(\sigma)}{\C'})>0 $. Hence  $  \sigma\in \Lk{\sigma}{\C}, $ this implies $ \sigma'=q(\sigma)\in q(\Mh{\C}). $ Therefore  $  \Mh{\C'}\subseteq q(\Mh{\C}). $  Hence the result.
	\\ \no The proof of the second part is similar to the proof of first part and follows from Proposition \ref{remlinkduplicate} and Lemma \ref{lemlinkneeded}.
\end{proof}

\no \textbf{Relationship between SR rings of $ \Delta(\C) $ and $ \Delta(\C') $}.
We first discuss the relationship of the SR Ideal in the following lemma.  
\begin{lemma}
	Let $ \C $ be a code on $n$ neurons and $ q:\C\to \C' $ be the duplicating code map described above. Then the SR ideal of $ \Delta(\C) $ is contained in the SR ideal of $ \Delta(\C') $, i.e., $ I_{\Delta(\C)} \subseteq I_{\Delta(\C')}.$ \label{lemsrdup}
\end{lemma}
\no The proof is similar to first part of Lemma \ref{lemsradd}.
\begin{theorem}
	Let $ \C $ be a code on $n$ neurons and $ q:\C\to\C' $ be the duplicating code map described above. Then  SR ring of $ \Delta(\C) $ is embedded in SR ring of $ \Delta(\C') $. \label{SRaddindup}
\end{theorem}
\no Proof of this theorem is similar to the proof of Theorem \ref{SRaddtrv}.
\\ \no Here too, we could not draw any conclusions about $ \Mh{\C} $ using the SR ring. \\
\begin{theorem}
	Let $ \phi:\ring{\D}\to\ring{\C} $ be a neural ring isomorphism and let $ q_\phi:\C\rar\D $ be the associated code map. If $ \Mh{\C}\subseteq \C $ then $ q_\phi(\Mh{\C})=\Mh{\D}. $
\end{theorem}
\no The proof of this theorem comes as a consequence of Theorem \ref{thmnycc1} and the results that we have shown in  Theorems \ref{thprem}, \ref{thadd} and \ref{thdup}.	

\section{Surjective neural ring homomorphism} 
\subsection{Projection} \label{projection}
Let $ \C $ be a code on $ n $ neurons and obtain  $ \C' $ from $ \C $ by deleting $ n^{\text{th}} $ neuron of each codeword in $ \C $, i.e.,  for any $ \sigma=\sigma_{1}^{}\cdots\sigma_{n-1}^{}\sigma_{n}^{}\in\C $ form $ \sigma'\in\C' $ such that $ \sigma'=\sigma_1^{}\cdots\sigma_{n-1}^{}. $ Define a map $ q:\C\to\C' $ with $ \sigma\mapsto\sigma'. $ We will call such a map $ q $ as projection map of $ \C $ onto $ \C'. $  Further extend the map to $\Delta(\C) $ and abusing the notation call the extended map as $ q. $ 
\begin{lemma}
	Let $ q$ be the extended projection map defined as above. Then $ q $ is a well defined surjective map from $\Delta(\C) $ to $ \Delta(\C')$.
	
\end{lemma}
\begin{proof}
	Given any $ \alpha=\alpha_1^{}\cdots\alpha_{n-1}^{}\alpha_n^{}\in\Delta(\C) $ there exists $ \beta=\beta_1^{}\cdots\beta_{n-1}^{}\beta_n^{}\in\C $ such that $ \alpha\subseteq\beta. $ Let $ \alpha', \beta' $ be the respective images of $ \alpha,\beta
	$ under the map $ q. $  Furthermore, $ \alpha'\subseteq\beta' $ and as $ q(\beta)\in\C' $ we have $\alpha'\subseteq \beta'=q(\beta)\in\C'. $ By definition of simplicial complex,  $ q(\alpha)=\alpha'\in\Delta(\C') $. Therefore the map $ q:\Delta(\C)\to\Delta(\C') $ given by $ \alpha\mapsto\alpha' $ is well defined.
	
	Next we show $ q $ is surjective. Let $ \alpha'\in\Delta(\C'). $ Then there exists $ \beta'\in\C' $ such that $ \alpha'\subseteq\beta'. $ We know that $ q:\C\to\C' $ is a surjective map. Therefore there exists $ \beta\in \C $ such that $ q(\beta)=\beta' $. Note that $ \beta=\beta'1 $ or $ \beta=\beta'0. $ In either case $ \beta'0\in\Delta(\C) $ and thus $ \alpha'0\subseteq\beta'0\in\Delta(\C). $ Now as $ \Delta(\C) $ is a simplicial complex, $ \alpha'0\in\Delta(\C) $. Moreover, $ q(\alpha'0)=\alpha'. $ Hence the map $ q:\Delta(\C)\to\Delta(\C') $ is surjective. 
\end{proof}
\begin{remark} \label{remprojimp}
	Given $ \sigma'\in\Delta(\C'). $ If $ \sigma'1\in\Delta(\C) $ then as $ \sigma'0\subseteq\sigma'1 $ we get $ \sigma'0\in\Delta(\C). $ Else if $ \sigma'1\notin\Delta(\C) $ then as $ q $ is surjective we must have $ \sigma'0\in\Delta(\C). $  Thus whenever $ \sigma'\in\Delta(\C'), $ we get that $ \sigma'0\in\Delta(\C). $   We will be using this remark in the proof of the following lemma. 
\end{remark}
\begin{lemma}
	Let $ \C $ be a code on $ n$ neurons and $ q:\C\to\C' $ be the projection map described above. Then for any $ \sigma'\in\Delta(\C') $ we have $ \Lk{\sigma'}{\C'}=q(\Lk{\sigma}{\C}), $ where $ \sigma=\sigma'0. $  \label{lemprojreq}
\end{lemma}
\begin{proof}
	Let $ \omega'\in\Lk{\sigma'}{\C'} $ and let $ \omega=\omega'0. $ We will show that $ \omega\in\Lk{\sigma}{\C}. $  As  $ \omega'\in
	\Lk{\sigma'}{\C'} $ we have $ \omega'\cap\sigma'=\emptyset. $ This implies $ \omega\cap\sigma=\omega'0\cap\sigma'0=\emptyset $.  Let $ \omega'\cup\sigma'=\tau', $ then $\tau'\in \Delta(\C')$ as $  \omega'\in
	\Lk{\sigma'}{\C'}. $ Moreover,  $ \tau'0\in\Delta(\C)$ (Refer Remark \ref{remprojimp}). Now consider $ \omega\cup\sigma=\omega'0\cup\sigma'0=\tau'0\in\Delta(\C). $ Therefore  $ \omega\in\Lk{\sigma}{\C}, $ which gives $ \omega'=q(\omega)\in q(\Lk{\sigma}{\C}). $   Hence $ \Lk{\sigma'}{\C'}\subseteq q(\Lk{\sigma}{\C}). $
	
	To prove the other side containment, consider $ \omega'\in q(\Lk{\sigma}{\C}). $ Then $ \omega'= q(\omega)	 $ for some $ \omega\in \Lk{\sigma}{\C}. $ Now either $ \omega=\omega'1 $ or $ \omega=\omega'0 $ and in both cases $ \omega'0\subseteq \omega\in\Lk{\sigma}{\C} $. Hence $ \omega'0\in\Lk{\sigma}{\C} $ as $ \Lk{\sigma}{\C} $ is a simplicial complex.   We will show that $ \omega'\in\Lk{\sigma'}{\C'}. $  Now as $ \omega'0\in\Lk{\sigma}{\C} $ we have $ \omega'0\cap\sigma=\emptyset. $ Also, as $ \omega'\cap\sigma'\subseteq \omega'0\cap\sigma'0=\omega'0\cap\sigma=\emptyset $, we have $ \omega'\cap\sigma' =\emptyset.$ Next, let $ \omega'\cup\sigma'=\tau'.$ Consider $\omega'0\cup\sigma=\omega'0\cup\sigma'0=\tau'0. $  Now as $ \omega'0\in\Lk{\sigma}{\C} $ we have $ \tau'0\in\Delta(\C) $ and $ q(\tau'0)=\tau'\in\Delta(\C'). $ So, $ \omega'\cup\sigma'\in\Delta(\C'). $ Therefore $ \omega'\in\Lk{\sigma'}{\C'}. $ Hence $ q(\Lk{\sigma}{\C})\subseteq\Lk{\sigma'}{\C'} $ and the result.
\end{proof}

Note that when $\sigma=\sigma'0\in\Delta(\C)$,  $ q(\Lk{\sigma}{\C}) \subseteq \Lk{\sigma}{\C} $. Reason: For any $ \alpha'\in q(\Lk{\sigma}{\C}) $, we have $ \beta\in\Lk{\sigma}{\C} $, such that $ \alpha'=q(\beta). $ Then either $ \beta $ is $ \alpha'0 $ or $ \alpha'1. $ In both cases, $ \alpha'0\in \Lk{\sigma}{\C}, $ it being a simplicial complex and $ \alpha'0\subseteq \alpha'1. $ Thus, \begin{equation}\label{eqproj}
	q(\Lk{\sigma}{\C}) \subseteq \Lk{\sigma}{\C}, \quad \text{ when }  \sigma=\sigma'0 \in \Delta(\C).
\end{equation} 
\no Using Lemma \ref{lemprojreq} we thus get $ \Lk{\sigma'}{\C} = 	q(\Lk{\sigma}{\C}) \subseteq\Lk{\sigma}{\C}.$ Next we try to connect the homological dimensions of $ \Lk{\sigma}{\C} $ and $ \Lk{\sigma'}{\C'} $ when $ \sigma'\in\Mh{\C'}. $ Before that we define closed star of a simplicial complex which we will use in further results and prove a lemma required.  

\begin{definition}[closed star]
	Let $ \sigma\in K $ be an element of a simplicial complex $ K. $ Define closed star of $\sigma $ as $ K|^\sigma=\Delta(\sk{\sigma}{K})=\Delta(\{\tau\in K\ |\ \sigma\subseteq\tau\}). $ \\ For example, consider $ \C=\{1,12,23\} $ then $ \Delta(\C)=\{\emptyset,1,2,3,12,23\}. $ We get $ \Delta(\C)|^{3}=\{\emptyset,2,3,23\}. $ Also, note that by the definition of simplicial complex of a code (Refer Definition \ref{defscofc}) we have that $K|^\sigma$ is a simplicial complex for any $\sigma\in K$. 
\end{definition}
\no Note that  if $\sigma'1$ is some element in $\Delta(\C)$ then $v\in\Lk{\sigma'0}{\C}$, where $v=\underbrace{00\cdots0}_{n-1\text{ times }}1$. Reason: It is clear that $v\cap\sigma'0=\emptyset.$ Also since $v\cup \sigma'0=\sigma'1$ and $\sigma'1\in\Delta(\C')$. We have that $v\in\Lk{\sigma'0}{\C}$. Therefore we can consider the closed star of $v\in\Lk{\sigma'0}{\C},$ which we require in the following lemma. 
\begin{lemma}
Let $ \C $ be a code on $ n$ neurons and $ q:\C\to\C' $ be the projection map described above. Let $ \sigma'1\in\Delta(\C) $. Then $ q(\Lk{\sigma'0}{\C}|^{v}) = \Lk{\sigma'1}{\C},$ where $ v =\underbrace{00\cdots0}_{n-1\text{ times}}1.$    \label{lemprojcre}
\end{lemma}
\begin{proof}	
Let $ \alpha\in\Lk{\sigma'1}{\C}, $ then $ \alpha $ must be of the form $\alpha'0 $ for some $ \alpha' $ on $ n-1  $ neurons as $ \alpha\cap\sigma'1
=\emptyset $. We will show that $ \alpha'\in q(\Lk{\sigma'0}{\C}|^{v}). $ We claim that $ \alpha'1\in\Lk{\sigma'0}{\C}. $ As $ \alpha'0=\alpha\in\Lk{\sigma'1}{\C} $ we have $ \alpha'0\cap\sigma'1=\emptyset. $ This gives $ \alpha'1\cap\sigma'0=\emptyset. $ Next, let $ \alpha'1\cup\sigma'0=\tau'1. 	 $ We need to show that $ \tau'1\in\Delta(\C). $ Now, $ \tau'1=\alpha'1 \cup \sigma'0= \alpha'0\cup \sigma'1 \in \Delta(\C) $ since $ \alpha'0
\in\Lk{\sigma'1}{\C}. $ Therefore $ \alpha'1\in\Lk{\sigma'0}{\C}. $  Next, as $ v\subseteq \alpha'1 $ we get $ \alpha'1\in\Lk{\sigma'0}{\C}|^{v}. $ Also as $ \alpha'0\subseteq \alpha'1 $ and $ \Lk{\sigma'0}{\C}|^{v} $ being a simplicial complex, yields $ \alpha'0\in \Lk{\sigma'0}{\C}|^{v}. $  
So, $ \alpha'=q(\alpha'0)\in q(\Lk{\sigma'0}{\C}|^{v}). $ Therefore, $ \Lk{\sigma'1}{\C}\subseteq q(\Lk{\sigma'0}{\C}|^{v}). $

For converse containment, consider $ \alpha'
\in q(\Lk{\sigma'0}{\C}|^{v}) $. Then there exists $ \beta\in\Lk{\sigma'0}{\C}|^{v} $ such that $ q(\beta)=\alpha'. $ However, $ \beta $ has only two forms, i.e, either $ \alpha'0 $ or $ \alpha'1. $ Now, if $ \beta=\alpha'1 $ then $ \alpha'1
\in\Lk{\sigma'0}{\C} $ by the definition of $ \Lk{\sigma'0}{\C}|^{v}. $ Or if $ \beta=\alpha'0 $ then as $ v\not\subseteq \alpha'0 $ we get that $ \alpha'1\in\Lk{\sigma'0}{\C}|^{v}. $  So, in either case $ \alpha'1\in\Lk{\sigma'0}{\C} $.  We need to show that $ \alpha=\alpha'0\in\Lk{\sigma'1}{\C}. $ As $ \alpha'1\in\Lk{\sigma'0}{\C} $ we have $ \alpha'1\cap\sigma'0=\emptyset. $ This implies $ \alpha'0\cap\sigma'1=\emptyset. $ Next, let $ \alpha'0\cup\sigma'1=\tau'1. $ Also, we have $ \tau'1=\alpha'1\cup\sigma'0\in \Delta(\C). $ Therefore $ \alpha'0\cup\sigma'1\in \Delta(\C). $ So, $ \alpha'0\in\Lk{\sigma'1}{\C}. $ Hence $  q(\Lk{\sigma'0}{\C}|^{v}) \subseteq \Lk{\sigma'1}{\C}$ and the result follows.
\end{proof}
\begin{lemma}\label{lemskn47}
Let $ \C $ be a code on $ n $ neurons, let $ q:\C\to\C' $ be the projection map described above, and let $ \sigma'\in\Mh{\C'} $.  Then there exists $ \sigma\in\Delta(\C) $ such that $ q(\sigma)=\sigma' $ and $ \dim \widetilde{H}_i( \Lk{\sigma}{\C})>0 $ for some $ i\geq 0. $ \label{lemlinkprojection}
\end{lemma}
\begin{proof}
%
Let $ \sigma'\in \Mh{\C'}. $ Then, by definition of $ \Mh{\C'},\ \sigma'\in\Delta(\C'). $ There are two possibilities: \\ 
\textbf{Case 1:} Let $ \sigma'1\notin\Delta(\C). $ Then by Remark \ref{remprojimp},  $ \sigma'0\in\Delta(\C). $ Let $ \sigma=\sigma'0. $ We first show that $ \Lk{\sigma}{\C}=\Lk{\sigma'}{\C'}. $   Let $ \omega=\omega_1\cdots\omega_n\in \Lk{\sigma}{\C} $ then we claim that  $ \omega_{n}=0.$ Suppose not, then $ v= \underbrace{00\cdots0}_{n-1\text{ times}}1\subseteq\omega\in\Lk{\sigma}{\C}.$ As $ \Lk{\sigma}{\C} $ is a simplicial complex we get $ v\in\Lk{\sigma}{\C}. $ This implies $ \sigma\cup v^{} \in\Delta(\C). $ However, $ \sigma'1= \sigma'0\cup \underbrace{00\cdots0}_{n-1\text{ times}}1=\sigma\cup v . $ This implies $ \sigma'1\in\Delta(\C) $ which is contradiction to the hypothesis that $ \sigma'1\notin\Delta(\C). $  Hence $ \omega_{n}=0 $ for all $ \omega\in\Lk{\sigma}{\C}. $ So, for all $ \omega\in\Lk{\sigma}{\C} $ we rewrite $ \omega=\omega'0, $ where $ \omega' $ is a codeword on $ n-1 $ neurons. Moreover, $ q(\omega)=\omega'=\omega'0=\omega. $ This gives us that $  \Lk{\sigma}{\C}\subseteq q(\Lk{\sigma}{\C}).$ The converse containment follows from Equation \ref{eqproj}, and we have $\Lk{\sigma}{\C}= q(\Lk{\sigma}{\C}). $ Further, by Lemma \ref{lemprojreq} we have $ \Lk{\sigma'}{\C'}=q(\Lk{\sigma}{\C})  $. Hence $ \Lk{\sigma'}{\C'}= \Lk{\sigma}{\C}.$ Therefore $ \dim \widetilde{H}_i (\Lk{\sigma}{\C})=\dim \widetilde{H}_i (\Lk{\sigma'}{\C'}) $ for all $ i. $ Also, as $ \sigma'\in\Mh{\C'} $ there exists  some $ j\geq 0$ such that $ \dim \widetilde{H}_j (\Lk{\sigma'}{\C'})>0. $ Hence, in this case $ \sigma=\sigma'0 $ belongs to $ \Delta(\C) $ with $ \dim\widetilde{H}_i (\Lk{\sigma}{\C})>0 $ for some $ i\geq 0. $ \\ \textbf{Case 2:} $ \sigma'1\in\Delta(\C)$.	
Let, if possible  $ \dim \widetilde{H}_i(\Lk{\sigma'0}{\C})=0 $ for all $ i\geq 0. $ Let $ v $ be the $ n^{\text{th}} $ vertex of $ \Delta(\C). $ 	As $ \Lk{\sigma'0}{\C}|^{v} \subseteq \Lk{\sigma'0}{\C} $ we get $  q(\Lk{\sigma'0}{\C}|^{v})\subseteq q(\Lk{\sigma'0}{\C})$. 
Next, as $ \sigma'\in\Mh{\C'}, $ there exists $ k\geq 0 $ such that $ \dim \widetilde{H}_k (\Lk{\sigma'}{\C'})>0. $ Therefore by Lemma \ref{lemprojreq}, $ \dim \widetilde{H}_k ( q(\Lk{\sigma'0}{\C}))>0. $   But, $ \dim\widetilde{H}_i(\Lk{\sigma'0}{\C})=0 $ for all $ i\geq 0 $, so we can conclude that the $ k^{\text{th}} $ dimensional hole that is in  $ q(\Lk{\sigma'0}{\C}) $ is formed by removal of $ v. $ So, we get $ \dim \widetilde{H}_k ( q(\Lk{\sigma'0}{\C}|^{v}))>0. $ Moreover, by Lemma \ref{lemprojcre}, $ \dim \widetilde{H}_k ( \Lk{\sigma'1}{\C'})>0. $ So, $ \sigma'1\in\Mh{\C} $. Hence the result.  
\end{proof}
\begin{theorem}
Let $ \C $ be a code on $ n  $ neurons and $ q:\C\to\C' $ be the projection map described above. Then $  \Mh{\C'}\subseteq q(\Mh{\C}). $  
\end{theorem}
\begin{proof}
Let $ \sigma'\in\Mh{\C'} $ then by Lemma \ref{lemlinkprojection} there exists $ \sigma\in\Delta(\C) $ such that $ q(\sigma)=\sigma' $ and there exists some $ i\geq 0 $ with $ \dim \widetilde{H}_i(\Lk{\sigma}{\C})>0. $ This implies $ \sigma\in\Mh{\C} $. Therefore  $ \sigma'=q(\sigma)\in q(\Mh{\C}). $ Hence the proof. 
\end{proof}
\begin{eg} \label{exprojection}
We give an example to show that $ q(\Mh{\C})$ need not be contained in $ \Mh{\C'} $.
Let $ \C=\{123,24,2\},  $ then $ \Delta(\C)=\{1,2,3,4,12,23,13,24,123\}. $ Projecting $ \C  $ on the first three neurons (deleting the $ 4^{\text{th}} $ neuron) we get $ \C'=\{123,2\}. $ Also,  $ \Delta(\C')=q(\Delta(\C))=\{1,2,3,12,23,13,123\}. $ Computing $ M_H $ on the mathematical software Macaulay2 (Refer \cite{M2}) we obtain $ \Mh{\C}=\{123,24\} $ and $ \Mh{\C'}=\{123\}. $ But $ q(\Mh{\C})=\{123,2\}. $ Therefore $ q(\Mh{\C})\not\subset \Mh{\C'}.  $
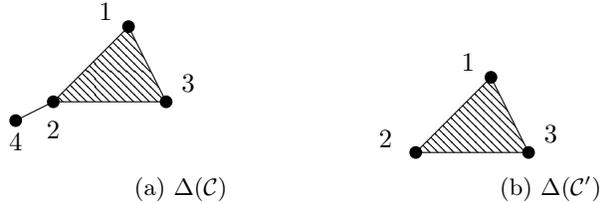
\begin{figure}[H]
	\begin{subfigure}[b]{0.3\linewidth}
		\begin{tikzpicture}[scale=0.5]
			\tikzstyle{point}=[circle,thick,draw=black,fill=black,inner sep=0pt,minimum width=4pt,minimum height=4pt]
			\node (a)[point, label={[label distance=-0.7cm]:$2$}] at (0,0) {};
			\node (b)[point,label={[label distance=0cm]5:$3$}] at (3,0) {};
			\node (c)[point,label={[label distance=-0.6cm]4:$1$}] at (2,2) {};

			\node (d)[point,label={[label distance=-0.6cm]:$4$}] at (-1,-0.5) {};

			
			\draw[pattern=north west lines] (a.center) -- (b.center) -- (c.center) -- cycle;
			\draw (a.center) -- (d.center);
			
		\end{tikzpicture}
		\caption{$ \Delta(\C) $}
	\end{subfigure}
	\centering
	\begin{subfigure}[b]{0.3\linewidth}
		\begin{tikzpicture}[scale=0.5]
			\tikzstyle{point}=[circle,thick,draw=black,fill=black,inner sep=0pt,minimum width=4pt,minimum height=4pt]
			\begin{scope}[xshift=2cm]
				\node (a)[point, label={[label distance=-0.7cm]5:$2$}] at (0,0) {};
				\node (b)[point,label={[label distance=0cm]5:$3$}] at (3,0) {};
				\node (c)[point,label={[label distance=-0.6cm]4:$1$}] at (2,2) {};
				\draw[pattern=north west lines] (a.center) -- (b.center) -- (c.center) -- cycle;
			\end{scope}	
			

			
			
			
		\end{tikzpicture}
		\caption{$ \Delta(\C') $}
	\end{subfigure}
\end{figure}
\end{eg}

\no\textbf{Relationship between SR rings of $ \Delta(\C) $ and $ \Delta(\C') $}.
\begin{theorem}
Let $ \C $ be a code on $n$ neurons and $ q:\C\to\C' $ be the projection map described above. Then the SR ring of $ \Delta(\C') $ is isomorphic to some sub-ring of the SR ring of $ \Delta(\C) $. \label{SRproj}
\end{theorem}
\begin{proof}
Define a code $ \C'' $ by adding a trivial off neuron and consider $ q':\C'\to \C'' $ to be the map obtained in Remark \ref{remaddtrvoff}. Then we observe that $ I_{\Delta(\C'')}\subseteq I_{\Delta(\C)}\subseteq  \mathcal{P}(n).  $ Therefore, the SR ring of $ \Delta(\C'')$ is a subset of SR ring of $ \Delta(\C), $ i.e., $  \mathcal{P}(n)/I_{\Delta(\C'')}\subseteq  \mathcal{P}(n)/ I_{\Delta(\C)}. $ Moreover, Remark \ref{remaddtrvoff} gives us that $  \mathcal{P}{(n-1)}/I_{\Delta(\C')}\cong  \mathcal{P}(n)/I_{\Delta(\C'')}. $ Hence we get that the SR ring of $ \Delta(\C') $ is isomorphic to some sub-ring of SR ring of $ \Delta(\C).  $
\end{proof}
\no In this case too, we could not draw any conclusions about $ \Mh{\C} $ using the SR ring since we do not get any nice relationships between the SR rings, unlike in a few previous cases. \\
\section{Future work}
A Jeffs \cite{jeffs2018homomorphisms} introduced new set of morphisms between two neural rings that helps in classifying open convex codes. We call these morphisms as \textit{trunk preserving morphisms}. We have seen that the neural ring homomorphisms are a subclass of the trunk preserving morphisms. We would leave it as open question to check in general, the behavior of trunk preserving morphisms on $\Mh{\C}$. 

	
Authors Affiliation:  \vspace{0.2cm}\\
Neha Gupta \\ Assistant Professor \\ Department of Mathematics \\ Shiv Nadar Institution of Eminence (Deemed to be University) \\ Delhi-NCR \\ India \\ Email: neha.gupta@snu.edu.in  \vspace{0.6cm}\\
 Suhith K N \\ Research Scholar \\ Department of Mathematics \\ Shiv Nadar Institution of Eminence (Deemed to be University) \\ Delhi-NCR \\ India \\ Email: sk806@snu.edu.in \\
 Suhith's research is partially supported by Inspire fellowship from DST grant IF190980.
\end{document}